\documentclass[12pt]{amsart}
\usepackage{geometry} 
\newcounter{dummy} \numberwithin{dummy}{section}
\newtheorem{theorem}[dummy]{Theorem}
\newtheorem{definition}[dummy]{Definition}
\newtheorem{proposition}[dummy]{Proposition}
\newtheorem{lemma}[dummy]{Lemma}
\newtheorem{remark}[dummy]{Remark}
\newtheorem{fact}[dummy]{Fact}
\newtheorem{bijection}[dummy]{Bijection}
\usepackage{graphicx}
\usepackage{amssymb}
\usepackage{epstopdf}
\usepackage{lscape}
\usepackage{indentfirst}
\usepackage{latexsym}
\usepackage{multirow}
\usepackage{tabls}
\usepackage{amsthm}
\usepackage{amsthm,amsmath}
\usepackage{graphicx}
\usepackage{indentfirst}
\usepackage{latexsym}
\usepackage{multirow}
\usepackage{tabls}
\usepackage{wrapfig}
\usepackage[all]{xy}
\usepackage[numbers]{natbib}
\usepackage{tikz}
\usetikzlibrary{positioning,shapes.misc}
\usepackage{enumitem}
\usepackage{subcaption}

\tikzset{
    myboxcircle/.style={circle,draw=black,align=center},
}
\tikzset{
    myboxrounded/.style={rounded rectangle,draw=black,align=center},
}
\tikzset{
    myboxrectangle/.style={rectangle,draw=black,align=center},
    myarrow/.style={->, >=latex', shorten >=1pt, thick},
    mylabel/.style={text width=7em, text centered} 
}
   \tikzstyle{block} = [draw,text centered, minimum height=0.8em,minimum width=2em]
        \tikzstyle{block1} = [text centered, minimum height=2.8em,minimum width=12em] 
\usepackage{mathrsfs}
\newtheorem{corollary}[dummy]{Corollary}
\usepackage{mathtools}

\usepackage[OT2,T1]{fontenc}
\DeclareSymbolFont{cyrletters}{OT2}{wncyr}{m}{n}
\DeclareMathSymbol{\Sha}{\mathalpha}{cyrletters}{"58}

\usepackage{tikz-cd}

\usepackage{float}
\counterwithin{table}{section}

\usepackage{multicol}

\title[On the Selmer group and rank of a family of elliptic curves]{On the Selmer group and rank of a family of elliptic curves and curves of genus one violating the Hasse principle}
\author{eleni agathocleous}
\address{\parbox{\linewidth}{Max-Planck-Institut f\"ur Mathematik, \\ Vivatsgasse 7, D-53111 Bonn, Germany.}}
\email{agathocleous@mpim-bonn.mpg.de}
\address{\parbox{\linewidth}{Center for Quantum Technology and Applications (CQTA), \\ Deutsches Elektronen-Synchrotron DESY, \\ Platanenallee 6, 15738 Zeuthen, Germany.}}
\email{eleni.agathocleous@desy.de}

\begin{document}
\begin{center}
\maketitle{\textbf{Abstract \footnote[1]{This work has been supported by the European Union’s H2020 Programme under grant agreement number ERC-669891, and partially supported with funds from the Ministry of Science, Research and Culture of the State of Brandenburg within the Centre for Quantum Technology and Applications (CQTA).}}}
%\tableofcontents
\end{center}

\tiny
We study an infinite family of $j$-invariant zero elliptic curves $E_{D}:y^{2}=x^{3}+16D$ and their $\lambda$-isogenous curves $E_{D'}:y^{2}=x^{3}-27\cdot16D$, where $D$ and $D' = -3D$ are fundamental discriminants of a specific form, and $\lambda$ is an isogeny of degree $3$. A result of Honda guarantees that for our discriminants $D$, the quadratic number field $K_{D} = \mathbb{Q}(\sqrt{D})$ always has non-trivial 3-class group. We prove a series of results related to the set of rational points $E_{D'}(\mathbb{Q}) \setminus \lambda(E_{D}(\mathbb{Q}))$, and the $SL(2,\mathbb{Z})$-equivalence classes of irreducible integral binary cubic forms of discriminant $D$. By assuming finiteness of the Tate-Shafarevich group, we derive a parity result between the rank of $E_{D}$ and the rank of its $3$-Selmer group, and we establish lower and upper bounds for the rank of our elliptic curves. Finally, we give explicit classes of genus-$1$ curves that correspond to irreducible integral  binary cubic forms of discriminant $D=48035713$, and we show that every curve in these classes violates the Hasse Principle.

\

\textbf{AMS Mathematics Subject Classification:} Primary: 11R29, 11G05. 

\

\textbf{Keywords:} Elliptic Curves, Selmer group, Rank, Ideal Class Group, Tate-Shafarevich, Hasse Principle.

\  

\normalsize
\section{Introduction}\label{intro}   The relation between the Selmer group of an elliptic curve and the ideal class group of an underlying quadratic number field is well known, and it is present in many papers related to descent on elliptic curves, such as \cite{Crem1}, \cite{Crem3} and \cite{SchaeferStoll}. The particular relation between the $3$-part of the ideal class group and the $3$-Selmer group of elliptic curves, has received great attention by many authors, as one can see for example in the works of \cite{Bandini}, \cite{Quer}, \cite{Satge}, \cite{Top}, and more recently of \cite{Bhar1} and \cite{Bhar2}. Satg\'e \cite{Satge} in particular, studied the Selmer group of the general family of elliptic curves of the form $y^2 = x^3+k$, relative to a 3-isogeny, and showed that the rank of the Selmer group depends heavily on the decomposition of $k$. 

In this paper, we study a specific family of such curves of $j$-invariant zero, namely the elliptic curves $E_{D}: y^2 = x^3 +16D$, where $D$ is a squarefree integer of the form $D = D_{m,n} = 4m^3 -27n^2$ and the corresponding cubic polynomial $f_{m,n}(x) = x^{3}-mx+n$ is irreducible. The curves $E_{D}$ have an isogeny $\lambda$ of degree $3$, and their $\lambda$-isogenous curves $E_{D'}$ are of the form $E_{D'}: y^2 = x^3 -27\cdot16D$, with $D' = -3D$ also a fundamental discriminant since $3 \nmid D$. A result of Honda \cite{Honda2} guarantees that for our discriminants $D$, the quadratic number field $K_{D} = \mathbb{Q}(\sqrt{D})$ has non-trivial 3-class group. Our main focus is the study of a direct relation between the elements of the quotient group $E_{D'}(\mathbb{Q})/\lambda(E_{D}(\mathbb{Q}))$ under the Fundamental $3$-descent map, and the unramified cubic extensions of $K_{D}$, emerging from the so-called $3$-virtual units of the quadratic resolvent $K_{D'} = \mathbb{Q}(\sqrt{-3D})$ of $K_{D}$. The topic and the main results of each section, are as follows:

In Section~\ref{CubicExtns} we discuss known results related to the unramified cubic extensions of the quadratic number fields $K_{D}$, in relation to the so-called $3$-virtual units of $K_{D'}$. Via a direct application of the results 
of Erd\"os~\cite{Erdos}, we show in Lemmas~\ref{lm:InfDP}~and~\ref{lm:InfPol} that the set $\mathcal{D}$ containing all these odd fundamental discriminants $D_{m,n}$ whose corresponding cubic polynomial $f_{m,n}(x) = x^{3}-mx+n$ is irreducible, is an infinite set, and therefore the corresponding family of elliptic curves $\{E_{D} \ | \ D\in \mathcal{D}\}$ is infinite. At the same time, this result establishes the existence of an infinite family of cubic fields $L$, namely those generated by the polynomials $f_{m,n}$, which are monogenic; i.e. their maximal order $\mathcal{O}_{L}$ is of the form $\mathbb{Z}[\alpha]$, for some $\alpha \in \mathcal{O}_{L}$. The property of monogenicity is directly related to the known Hasse's Problem from the 1960's regarding the classification of algebraic number fields (see for example \cite{MotodaEtal} and \cite{Nakahara}), and many authors study this problem in its own right (see for example \cite{Boydetal}, \cite{Ibarraetal} and \cite{Kedlaya}). The study of cubic monogenic fields has recently gained increased attention, particularly due to their direct connection with the rank of the elliptic curves $y^{2}=x^{3}+k$ \cite{AlpogeBhargShni}. This direct connection is a key focus of our paper, and in the subsequent sections we explicitly demonstrate through a series of results, how monogenic orders of the cubic fields $L$ of discriminant $D$, whose index in the maximal order $\mathcal{O}_{L}$ is an integral cube, provide a lower bound for the rank of the elliptic curves $E_{D}$. 

In Section~\ref{EllCurvesED}, we discuss known results on descent by $3$-isogeny adjusted to the elliptic curves $E_{D}$. In Proposition~\ref{Selmer}, via the results of Satg\'e \cite{Satge}, we prove that the ranks of the Selmer groups $\mathcal{S}_{\lambda}$ and $\mathcal{S}_{\lambda'}$, with respect to the isogenies $\lambda$ and $\lambda'$, are always equal to the following specific values associated with the $3$-rank of the ideal class group of $K_{D}$, which we denote by $r_{3}(D)$:
\begin{proposition}
For every $D \in \mathcal{D}$, the rank of the Selmer groups $\mathcal{S}_{\lambda}$ and $\mathcal{S}_{\lambda'}$ of the curves $E_{D}$ and $E_{D'}$ are as follows: 
$$d(\mathcal{S}_{\lambda}) = d(\mathcal{S}_{\lambda'}) = r_{3}(D), \ \ \ \ \text{if $D < -4$}$$ 
$$d(\mathcal{S}_{\lambda}) = r_{3}(D) \ \text{and} \ d(\mathcal{S}_{\lambda'}) = r_{3}(D) + 1, \ \ \ \  \text{if $D > 0$}.$$ 
\end{proposition}
In the same section, by employing Voronoi's algorithm, we prove in Proposition~\ref{prop:CubeCoeff} that a point $P$ lies in the set $E_{D'}(\mathbb{Q}) \setminus \lambda(E_{D}(\mathbb{Q}))$ if and only if there exists an irreducible integral binary cubic form of discriminant $D$ having leading coefficient an integral cube, and which defines the cubic field $L$ associated to the point $P$. The direct relation between these points $P$ and their associated cubic field $L$ is proved in Lemma~\ref{lem:basis}. 

In Section~\ref{sec:rank}, by assuming finiteness of the Tate-Shafarevich group and by employing Proposition~1.1 from above, we derive in Corollary~\ref{ParityRanksDiagram} a parity result between the rank of $E_{D}$ and the sum of the ranks of its two Selmer groups $\mathcal{S}_{\lambda}$ and $\mathcal{S}_{\lambda'}$. This corollary, together with Proposition~\ref{Rank} that gives lower and upper bounds for the rank of our elliptic curves, constitute the main result of the paper. The lower bound for the curves $E_{D}$ is directly related to the rank of a group $\mathcal{M}$ that is a subgroup of the $3$-Selmer group of $K_{D'}$, by definition. This group $\mathcal{M}$ is generated by all those $3$-virtual units, which give rise to cubic fields that can be defined by a binary cubic form of discriminant $D$ and with leading coefficient a cube. By our Proposition~\ref{prop:CubeCoeff}, the group $\mathcal{M}$ corresponds to the so-called \emph{solvable} part of the Selmer group $\mathcal{S}_{\lambda}$ of the elliptic curve $E_{D}$, hence its rank $\mathfrak{r}_{3}$ as an $\mathbb{F}_{3}$-vector space serves as a lower bound. In the proposition below we give the lower and upper bounds that we have established for the curves $E_{D}$. The third statement of the proposition is the parity result that we proved in the aforementioned Corollary~\ref{ParityRanksDiagram}.
\begin{proposition} 
For every $D \in \mathcal{D}$, by assuming finiteness of the $3$-primary part $\Sha(E_{D})[3^{\infty}]$ of the Tate-Shafarevich group $\Sha(E_{D})$ of the elliptic curves $E_{D}$, the rank of the elliptic curves $E_{D}$, and therefore of $E_{D'}$, is bounded as follows: 
\begin{enumerate}
\item If $\mathfrak{r}_{3}$ is odd, then
$$(i) \  \ 2 \leq  \mathfrak{r}_{3} + 1 \leq rank(E_{D}) \leq 2r_{3}(D), \ \text{if} \ D < -4$$  
$$(ii) \ \ 1 \leq \mathfrak{r}_{3} \leq rank(E_{D}) \leq 2r_{3}(D) +1, \ \text{if} \ D > 4.$$
\item If $\mathfrak{r}_{3}$ is even, then
$$(i) \  \ 2 \leq  \mathfrak{r}_{3} \leq rank(E_{D}) \leq 2r_{3}(D), \ \text{if} \ D < -4$$  
$$(ii) \ \ \mathfrak{r}_{3} +1 \leq rank(E_{D}) \leq 2r_{3}(D) +1, \ \text{if} \ D > 4.$$
\item $rank(E_{D})$ is always even for negative $D \in \mathcal{D}$ and is always odd for positive $D \in \mathcal{D}$.
\end{enumerate}
\end{proposition}

The explicit knowledge of the elements of the Selmer group associated to the isogeny $\lambda$ that we established in the previous sections, allows us in Section~\ref{Torsors} to describe the classes of \emph{solvable} Torsors associated with the Selmer group $\mathcal{S}_{\lambda}$. In Proposition~\ref{ViolateHassePrinciple} in particular, we give an if-and-only-if condition on whether specific genus-$1$ curves violate the Hasse Principle. In Section~\ref{subs:solvability} we discuss the \emph{solvable} part of $\mathcal{S}_{\lambda}$ in relation to the open question of the uniform boundedness of the ranks of elliptic curves over $\mathbb{Q}$, and we address computational aspects related to the lower bound associated with $\mathfrak{r}_{3}$. 

Finally, in Section~\ref{sec:examples} we give explicit classes of genus-$1$ curves that correspond to irreducible integral  binary cubic forms of discriminant $D=D_{229,3}$, and we show that every curve in these classes violates the Hasse Principle.

\section{Unramified cubic extensions of quadratic number fields}\label{CubicExtns}
\subsection{The $3$-part of the ideal class group.}
Throughout the paper, $disc(\cdot)$ will denote the integer that is equal to the discriminant of the object in the parentheses. 

Honda, in \cite[Proposition 10]{Honda2}, showed the following important result:
\begin{proposition}\label{PropHonda} \emph{\cite[Proposition 10]{Honda2}} If the class number of a quadratic number field $K$ is a multiple of $3$, $K$ must be of the form 
$$K_{\alpha,\beta} = \mathbb{Q}(\sqrt{4\alpha^3-27\beta^2}),$$ for some $\alpha, \beta \in \mathbb{Z}$. Conversely, if for integers $\alpha, \beta$ we have $gcd(\alpha,3\beta)=1$ and if $\alpha$ cannot be represented by a form $(\beta+c^{3})c^{-1}$ for any $c \in \mathbb{Z}$, then $3 | h_{K_{\alpha,\beta}}$, the class number of $K_{\alpha,\beta}$. $\hfill \square$  \end{proposition} 

We note that in Honda's result, $$4\alpha^3-27\beta^2 = \mathfrak{f}^{2} disc(K_{\alpha,\beta}),$$ where $disc(K_{\alpha,\beta})$ is the fundamental discriminant of $K_{\alpha,\beta}$ and $\mathfrak{f}\geq 1$ is some conductor. Furthermore, the condition on $\alpha$ not being represented by the given form, for any $c \in \mathbb{Z}$, implies that the cubic polynomial $$f_{\alpha,\beta}(x) = x^{3}-\alpha x+\beta \in \mathbb{Z}[x]$$
is irreducible. In the same paper \cite{Honda2}, Honda showed that the irreducible cubic polynomials of the form $f_{\alpha,\beta}$ define the unramified cubic extensions of $K_{\alpha,\beta}$, and the Galois group of their Galois closure, equivalently, the Galois group of their splitting field 
\begin{equation}\label{splitting} L_{\alpha,\beta} \supsetneqq K_{\alpha,\beta} \supsetneqq \mathbb{Q},\end{equation} is isomorphic to the symmetric group $\mathbb{S}_{3}$.

\subsection{Cubic fields of fundamental discriminant.}\label{FieldsDiscD}
Denote by $L_{d}$ the algebraic number fields of degree $d$ whose splitting field $Spl_{d}$ has Galois group isomorphic to the symmetric group $\mathbb{S}_{d}$. For degree $d \leq 5$, Kondo in \cite[\S 4]{Kondo} and also Yamamura in  \cite[Remark on pg. 107]{Yamamura} showed the following:

\begin{theorem}\label{disc} \emph{(\cite[\S 4]{Kondo} \& \cite[Remark on pg. 107]{Yamamura})} Assume that $Spl_{d}/K$ is unramified at all finite primes, where $K$ is a quadratic subfield. Then, if $d \leq 5$, the set of the non-conjugate algebraic number fields $L_{d}$ equals the set of the non-conjugate algebraic number fields of degree $d$ with discriminant equal to that of the quadratic number field $K$. $\hfill{\square}$\end{theorem}     

For $d=3$, the unramified cubic extensions $Spl_{d}/K$ correspond to the unramified cubic extensions of Honda described in Proposition~\ref{PropHonda} above, with the underlying quadratic number fields $K$ being of the form $K_{\alpha,\beta}$. The cubic fields $L_{3}$ have discriminant $disc(L_{3}) = disc(K_{\alpha,\beta})$, they can be defined by a cubic polynomial of the form $f_{\alpha,\beta}$, and they identify with the splitting fields $L_{\alpha,\beta}$ of Relation~(\ref{splitting}). 

\subsection{Our Discriminants $D$}\label{subs:discs} 
For any pair of integers $m, n$ we compute the quantity $D_{m,n}=4m^{3}-27n^{2}$. Whenever $D_{m,n}$ is squarefree we include it in the set 
$$\mathcal{D}'\coloneqq \{ D_{m,n} \ | \ m,n \in \mathbb{Z} \ \text{and} \ D \ \text{is squarefree}\},$$
and the corresponding pair of integers $(m,n)$ we include it in the set 
$$\mathcal{P}'\coloneqq \{ (m,n) \in \mathbb{Z}^{2}\ | \ D_{m,n} \in \mathcal{D}' \}.$$

\begin{remark}
  Since $D_{m,n} \in \mathcal{D}'$ is squarefree, the following are straightforward observations:
  \begin{enumerate}
      \item  $\emph{gcd}(2m,3n)=1$, yielding in particular that $n$ is odd, and therefore
      \item $D_{m,n} = 4m^{3}-27n^{2} \equiv 1 \bmod 4.$
  \end{enumerate}
\end{remark}
We show below, via a direct application of a theorem of Erd\"os, that the set $\mathcal{D}'$, and therefore the set $\mathcal{P}'$, is infinite. 
\begin{lemma}\label{lm:InfDP}
    The sets $\mathcal{D}'$ and $\mathcal{P}'$ are infinite.
\end{lemma}
\begin{proof}
For any fixed odd integer $n_{0}$, we consider the cubic polynomial $$h_{n_{0}}(x) = 4x^{3}-27n_{0}^{2}.$$ It is straightforward to see that there are no linear polynomials $g_{1}, g_{2}$ with integral coefficients such that $h_{n_{0}} = g_{1}^{2}g_{2}$. Therefore, following directly from Erd\"os' Theorem \cite[Theorem, pg. 417]{Erdos}, we have that there are infinitely many integers $m$ for which the value $h_{n_{0}}(m)$ is squarefree. 
\end{proof}

\begin{remark}\label{rmrk:density}
If we fix an integer $m_{0}$ coprime to $3$, we can look instead at the quadratic polynomials $g_{m_{0}}(x) = 27x^2-4m_{0}^3$. Again, it is easy to see that $g_{m_{0}}$ cannot be written as a square of a linear polynomial with integral coefficients. Then, as Erd\"os mentions in the same paper \cite[pp. 416 - 417]{Erdos}, it has been known in this case as well that there are infinitely many integers $n$ for which $g_{m_{0}}(n)$ is squarefree. Moreover, in this case, the integers $n$ for which $g_{m_{0}}(n)$ is squarefree, have positive density. 
\end{remark}

Our discriminants $D_{m,n}$ are fundamental. Hence, the ones whose corresponding cubic polynomial $$f_{m,n}(x) = x^{3} - mx + n$$ is irreducible, are precisely the Honda discriminants with conductor $\mathfrak{f}=1$ and with  $Gal(L_{m,n}/\mathbb{Q}) \cong S_{3}$, where the fields $L_{m,n}$ are the splitting fields of Relation~(\ref{splitting}) corresponding to these fundamental discriminants $D_{m,n}$. 

\begin{lemma}\label{lm:InfPol}
    The set of cubic polynomials
$$\mathcal{F}\coloneqq \{ f_{m,n} \ | \ f_{m,n} \ \text{irreducible}, disc(f_{m,n}) \in \mathcal{D}', Gal(L_{m,n}/\mathbb{Q}) \cong S_{3}\}$$
is infinite.
\end{lemma}
\begin{proof}
For any pair $(m,n) \in \mathcal{P}'$, consider the corresponding polynomial $f_{m,n}$ and assume that it is reducible. Then, there is a monic linear integral polynomial such that 
$$f_{m,n}(x) = (x + a)(x^{2} + Ax +B).$$ Elementary computations show that in this case we should have 
$$a = -A, \ m = B-A^{2}, \ \text{and} \ n=-AB.$$ Since $n$ is always odd, because $(m,n) \in \mathcal{P}'$,   both $A$ and $B$ should also be odd, which implies that if $f_{m,n}$ is reducible, then $m$ must be even.

Then, for any fixed $m_{0}$ odd, Remark~\ref{rmrk:density} gives us the existence of infinitely many fundamental discriminants $D_{m_{0},n}$ and infinitely many corresponding irreducible cubic polynomials $f_{m_{0},n}$.
\end{proof}

Given the set  $\mathcal{F}$ of Lemma~\ref{lm:InfPol}, we define below the sets 
$\mathcal{P}\subseteq \mathcal{P}'$ and $\mathcal{D}\subseteq \mathcal{D}'$, which by Lemma~\ref{lm:InfPol}, are also infinite:
$$\mathcal{D} \coloneqq \{ D_{m,n} \in \mathcal{D}' \ | \ f_{m,n} \in \mathcal{F} \} \subseteq \mathcal{D}',$$
$$\mathcal{P}\coloneqq \{ (m,n) \in \mathcal{P}' \ | \ f_{m,n} \in \mathcal{F} \} \subseteq \mathcal{P}'.$$
Our cubic polynomials $f_{m,n} \in \mathcal{F}$ define a subfamily of cubic fields, let us denote them by $L_{m,n}$, and since $disc(L_{m,n}) = D_{m,n} =disc(f_{m,n})$, this subfamily contains precisely those cubic fields $L_{3}$ of Theorem~\ref{disc} which are monogenic.

From now on and throughout the paper, the letter $D$ will always correspond to a discriminant $D_{m,n} \in \mathcal{D}$, for some pair of integers $(m,n) \in \mathcal{P}$, and with corresponding irreducible polynomial $f_{m,n} \in \mathcal{F}$. Whenever it is clear, we will lose the subscript and we will denote $D_{m,n}$ only by $D$. 

We will always denote by $K_{D}$ the quadratic number field $K_{D}=\mathbb{Q}(\sqrt{D})$, and by $K_{D'} = \mathbb{Q}(\sqrt{D'})$ its quadratic resolvent, where $D' = -3D$ is also a fundamental discriminant, since $3 \nmid D$. Furthermore, we always have that
$$D' = -3D \equiv 1 \bmod 4.$$

Finally, we will say that a polynomial of the form $f_{\alpha,\beta}$ is in \emph{standard form}, if there is no integer $c$ such that $c^{2} | \alpha$ and $c^{3} | \beta$.

\subsection{Unramified cubic extensions and $3$-virtual units}\label{sec:L[mu]}
In \cite[Chapter 5]{Cohen1}, for a given prime $l$, number field $K$ and ring of integers $O_{K}$, Cohen gives the definition of the group $V_{l}(K)$ of $l$-virtual units of $K$ and proves a number of important properties. We cite below what we need, by adjusting them to our case, for $l = 3$. 

\begin{proposition} \emph{\cite[Proposition 5.2.3]{Cohen1}} \label{fact1} Let $\mu \in K^{\times}$. The following two properties are equivalent:

(1) There exists an ideal $\mathfrak{a}$ such that $\mu O_{K} = \mathfrak{a}^3$. 

(2) The element $\mu$ belongs to the group generated by the units, the cube powers of elements of $K^{\times}$ and the cube powers of nonprincipal ideals which become principal when raised to the cube power, hence they belong to the $3$-torsion part $Cl(K)[3]$ of the ideal class group $Cl(K)$ of $K$. \end{proposition}

\begin{definition}\label{def3unit}\emph{\cite[Definition 5.2.4]{Cohen1}} \\ (1) An element $\mu \in K^{\times}$ satisfying one of the two equivalent conditions of the above proposition will be called a $3$-virtual unit. \\ (2) The set of $3$-virtual units forms a multiplicative group which we denote by 
$V_{3}(K)$. \\ (3) The quotient group $Sel_{3}(K) \coloneqq V_{3}(K)/{K^{\times}}^{3}$ will be called the $3$-Selmer group of $K$. \end{definition}
The following sequence \cite[Section 5.2]{Cohen} shows the relation between the Selmer group $Sel_{3}(K) = V_{3}(K)/{K^{\times}}^{3}$ of $K$, the $3$-part of its ideal class group, and the group of units $U_{K} \subset \mathcal{O}_{K}$: 
\begin{equation}\label{Selmershort} 1 \rightarrow \frac{U_{K}}{{U_{K}}^{3}} \rightarrow Sel_{3}(K) \xrightarrow{\Phi} Cl(K)[3] \rightarrow 1. \end{equation}
We denote by $[\mu] \coloneqq \mu \bmod (K^{\times})^{3}$ the class of a $3$-virtual unit $\mu = \mathfrak{a}^{3}$ in $Sel_{3}(K)$. The map $\Phi$ sends the class $[\mu] \in Sel_{3}(K)$, to the ideal class $[\mathfrak{a}] \in Cl(K)[3]$. The pre-image of $\Phi$ is unique, up to a unit. More specifically, 
$$\Phi^{-1}([\mathfrak{a}]) = \{[\mu], \varepsilon [\mu],\varepsilon^{2} [\mu]\}, \ \text{where} \ \varepsilon \ \text{is the fundamental unit of} \ K.$$ 
Of course, if $disc(K) < -4$, there are no units other than $\pm 1$ and $\Phi$ is an isomorphism.

We recall that an ideal of $\mathcal{O}_{K}$ is called \emph{primitive} if there exists no rational integer $k\neq \pm 1$ such that every element of the ideal is a multiple of $k$. A $3$-virtual unit $\mu \in \mathcal{O}_{K}$ will be called \emph{primitive} if $\mu\mathcal{O}_{K} = \mathfrak{m}$ for some primitive $\mathcal{O}_{K}$-ideal $\mathfrak{m}$.

It is known that every cubic field $L$ of fundamental discriminant $\Delta$ arises from a $3$-virtual unit in the quadratic resolvent $K_{\Delta'} = \mathbb{Q}(\sqrt{-3\Delta})$ of the quadratic number field $K_{\Delta} = \mathbb{Q}(\sqrt{\Delta})$, contained in the Galois closure of $L$. More specifically, \cite[Lemma 4.2]{Hambleton}, every such cubic field $L$ is generated by an irreducible cubic polynomial of the form
$$f_{\mu}(x) = x^{3} - 3(\mu \overline{\mu})^{1/3}x + (\mu+\overline{\mu}) \in \mathbb{Z}[x],$$ 
where $\mu$ is a primitive $3$-virtual unit $\mu \in \mathcal{O}_{\Delta'} \setminus \mathcal{O}_{\Delta'}^{3}$, with $\mathcal{O}_{\Delta'}$ being the maximal order of $K_{\Delta'}$. Two $3$-virtual units $\mu_{1} = \mathfrak{a}_{1}^{3}$ and $\mu_{2}= \mathfrak{a}_{2}^{3}$ are generators of the same field, up to conjugation, if and only if $\mu_{1}$ or $\overline{\mu_{1}}$ is equal to $\alpha^{3} \mu_{2}$, for some $\alpha \in K_{\Delta'}^{\times}$ \cite[Theorem 4.3]{Hambleton}. Equivalently, if and only if 
the two groups $\langle [\mu_{1}] \rangle$ and $\langle [\mu_{2}] \rangle$ generated by $[\mu_{1}], [\mu_{2}] \in Sel_{3}(K_{\Delta'})$, are equal. Therefore, a subgroup $\langle [\mu] \rangle \subseteq Sel_{3}(K)$ corresponds to a unique, up to conjugation, cubic field, which we will denote by $L_{[\mu]}$. 

For fundamental discriminants $\Delta$, a classical result of Hasse \cite[Satz 7, pg.578]{Hasse} gives us that the number of non-conjugate cubic fields of discriminant $\Delta$ is equal to $$\frac{(3^{r_{3}(\Delta)}-1)}{2},$$ where $r_{3}(\Delta)$ is the rank of the $3$-part of the ideal class group of $K_{\Delta}$.

We keep in mind that $Sel_{3}(K)$ is $3$-primary. Either every subgroup in $Sel_{3}(K)$, or all but one, corresponds to a cubic field of discriminant $disc(L) = \Delta$. Whether it is the first case or the second, it depends on the ranks of the $3$-part of the ideal class groups of $K_{\Delta}$ and its quadratic resolvent $K_{\Delta'}$. For more details, the reader may refer to \cite[\S 4.5]{Hambleton}.

\section{The elliptic curves $E_{D}$}\label{EllCurvesED}
\subsection{Descent by $3$-isogeny and the Selmer groups of $E_{D}$ and $E_{D'}$.}\label{CurvesAndSelmer}
For a detailed exposition of the theory and tools of this section, the reader may refer to \cite[Chapter X.4]{Silverman} for the material related to the Selmer and Tate-Shafarevich groups, and to \cite[Appendix B]{Silverman}, \cite[Chapter IV]{CassFr}  and \cite[Chapter IV, \S3]{Neuk} for Galois Cohomology. 

For any number field $M/\mathbb{Q}$, denote by $E(M)$ the group of points of the elliptic curve $E$ defined over $M$. Let $\overline{M}$ denote the algebraic closure of $M$. 

To every discriminant $D \in \mathcal{D}$ we attach the $j$-invariant zero elliptic curve $E_{D}$ defined as 
$$E_{D}: y^2 = x^3 +16D.$$
The subgroup of order 3 of $E_{D}(\mathbb{\overline{Q}})$ invariant under the action of $G_{\mathbb{Q}} = Gal(\overline{\mathbb{Q}}/\mathbb{Q})$, is the torsion group 
$$\mathcal{T}_{D} = \{ O, (0,\pm 4 \sqrt{D}) \} \subseteq E_{D}(\overline{\mathbb{Q}}).$$Therefore, the quotient of $E_{D}$ over $\mathcal{T}_{D}$ gives a family of isogenous elliptic curves, also defined over $\mathbb{Q}$, which we denote by $E_{D'}$, with $D' = -3D$. These curves $E_{D'}$ are of the form \begin{equation}\label{isogenous}  E_{D'}: Y^2 = X^3 - 27\cdot 16D,  \end{equation}
and they also have a torsion subgroup of order 3 invariant under $G_{\mathbb{Q}}$, namely $$\mathcal{T}_{D'} =   \{ O,(0,\pm 12 \sqrt{D'}) \} \subseteq E_{D'}(\overline{\mathbb{Q}}).$$
We remark that the subscripts $D$ and $D'=-3D$ in our notation for $E_{D}$ and $E_{D'}$, are consistent with the definition for the general equation of an elliptic curve with $3$-isogeny, as given for example in \cite[\S 8.4.2, pg. 558]{Cohen}. The discriminants $D$ and $D'$ constitute the squarefree part of the constant term of their respective equations.

We define by $\lambda$ the rational 3-isogeny $$\lambda: E_{D} \rightarrow E_{D'} $$ with kernel $ker_{\lambda}(\mathbb{\overline{Q}}) = \mathcal{T}_{D}$, given by the quotient map:
$$X = (y^2 +3\cdot16D)x^{-2} \ \text{and} \ Y = y(x^3-8\cdot16D)x^{-3} .$$
The dual isogeny $\lambda'$ is such that $\lambda \lambda' = \times 3$, where $\times 3$ is the multiplication-by-3 map. Since both $ker_{\lambda}(\mathbb{\overline{Q}})$ and $ker_{\lambda'}(\mathbb{\overline{Q}})$ contain no non-trivial rational point of order 3, we have that the rank of $E_{D}(\mathbb{Q})$ is \begin{equation}\label{R}rank(E_{D}(\mathbb{Q})) = dim_{\mathbb{F}_{3}}(E_{D}(\mathbb{Q})/3E_{D}(\mathbb{Q})).\end{equation} 
Consider the short exact sequence, \cite[Section X.4, Remark 4.7]{Silverman}: 
\begin{equation} \label{splitrank} \small 0 \rightarrow E_{D'}(\mathbb{Q})/\lambda(E_{D}(\mathbb{Q})) \xrightarrow{\lambda'} E_{D}(\mathbb{Q})/3E_{D}(\mathbb{Q}) \rightarrow E_{D}(\mathbb{Q})/\lambda'(E_{D'}(\mathbb{Q})) \rightarrow 0 .\end{equation} 
Given fact (\ref{R}) above, to compute the rank of $E_{D}$ it suffices to compute the dimensions of the quotients $E_{D'}(\mathbb{Q})/\lambda(E_{D}(\mathbb{Q}))$ and $E_{D}(\mathbb{Q})/\lambda'(E_{D'}(\mathbb{Q}))$.\normalsize
 
Consider now the short exact sequence \begin{equation}\label{lisogeny} 0 \rightarrow \mathcal{T}_{D} \rightarrow E_{D}(\mathbb{\overline{Q}}) \xrightarrow{\lambda} E_{D'}(\mathbb{\overline{Q}}) \rightarrow 0.\end{equation} From this, we obtain the long exact cohomology sequence which gives in particular the following 
\begin{equation}\label{quotient} 0 \rightarrow E_{D'}(\mathbb{Q})/\lambda(E_{D}(\mathbb{Q}))  \xrightarrow{\delta}  H^{1}(G_{\mathbb{Q}},\mathcal{T}_{D}) \rightarrow H^{1}(G_{\mathbb{Q}}, E_{D}(\mathbb{\overline{Q}}))[\lambda] \rightarrow 0.\end{equation} 
By localising at each place $p$ of $\mathbb{Q}$, we obtain the following commutative diagram, where $res_{p}$ is the usual restriction map: 
\small
\begin{center}
\begin{tikzcd}[ampersand replacement=\&, column sep=tiny]
        0 \arrow[r] \&  E_{D'}(\mathbb{Q})/\lambda(E_{D}(\mathbb{Q})) \arrow[r, "\delta"] \arrow[d]  \& H^{1}(G_{\mathbb{Q}},\mathcal{T}_{D}) \arrow[r]  \arrow[d, "\underset{p}{\prod}res_{p}"] \& H^{1}(G_{\mathbb{Q}}, E_{D}(\mathbb{\overline{Q}}))[\lambda]  \arrow[r] \arrow[d, "\underset{p}{\prod}res_{p}"]  \& 0 \\  
        0 \arrow[r] \&  \underset{p}{\prod}E_{D'}(\mathbb{Q}_{p})/\lambda(E_{D}(\mathbb{Q}_{p})) \arrow[r]  \& \underset{p}{\prod}H^{1}(G_{\mathbb{Q}_{p}},\mathcal{T}_{D}) \arrow[r]  \& \underset{p}{\prod}H^{1}(G_{\mathbb{Q}_{p}}, E_{D}(\mathbb{\overline{Q}}_{p}))[\lambda]  \arrow[r]  \& 0 
    \end{tikzcd}
\end{center}    
\normalsize

\begin{definition}\label{definition} The Selmer group of $E_{D}$ relative to the isogeny $\lambda$ is 
$$\mathcal{S}_{\lambda} = \{ \xi \in H^{1}(G_{\mathbb{Q}},\mathcal{T}_{D}) \ | \ res_{p}(\xi) \in Im( E_{D'}(\mathbb{Q}_{p})/\lambda(E_{D}(\mathbb{Q}_{p}))) \ \text{for every} \ p\}.$$

The Tate-Shafarevich group of $E_{D}$ can now be defined as $$\Sha(E_{D}) = \{\xi \in H^{1}(G_{\mathbb{Q}}, E_{D}(\mathbb{\overline{Q}})) \  | \ res_{p}(\xi) = 0 \ \text{for every} \ p\}.$$

These two groups are connected together as follows: 
\begin{equation} \label{selsha} 0 \rightarrow E_{D'}(\mathbb{Q})/\lambda(E_{D}(\mathbb{Q})) \rightarrow \mathcal{S}_{\lambda} \rightarrow \Sha(E_{D})[\lambda] \rightarrow 0.\end{equation} \end{definition}

\begin{remark}\label{dualisogenydefn} By considering the dual isogeny $\lambda'$ instead, we get exact sequences analogous to (\ref{lisogeny}) and (\ref{quotient}) which in turn give us the analogous definitions for
$\mathcal{S}_{\lambda'}$,  $\Sha(E_{D'})$ and $\Sha(E_{D'})[\lambda']$. \end{remark}

To study the quotient groups $E_{D'}(\mathbb{Q})/\lambda(E_{D}(\mathbb{Q}))$ and $E_{D}(\mathbb{Q})/\lambda'(E_{D'}(\mathbb{Q}))$, an obvious approach is to study $H^{1}(G_{\mathbb{Q}},\mathcal{T}_{D})$ and $H^{1}(G_{\mathbb{Q}},\mathcal{T}_{D'})$, which contain isomorphic copies of them. We will discuss here the results of Satg\'e \cite{Satge}, which are related to $H^{1}(G_{\mathbb{Q}},\mathcal{T}_{D})$ and which will be of use for us.

Let $G_{K_{D}} = Gal (\mathbb{\overline{Q}}/K_{D})$. Since $G_{K_{D}}$ acts trivially on $\mathcal{T}_{D}$, the inflation-restriction sequence gives: 
\small 
\begin{equation}\label{long2} 0 \rightarrow H^{1}(Gal(K_{D}/\mathbb{Q}), {\mathcal{T}_{D}}^{G_{K_{D}}}) \xrightarrow{inf} H^{1}(G_{\mathbb{Q}},\mathcal{T}_{D})\end{equation}
$$ \xrightarrow{res} H^{1}(G_{K_{D}},\mathcal{T}_{D})^{Gal(K_{D}/\mathbb{Q})}  \rightarrow  H^{2}(Gal(K_{D}/\mathbb{Q}),{\mathcal{T}_{D}}^{G_{K_{D}}}).$$
\normalsize
Since $|Gal(K_{D}/\mathbb{Q})| =2$ and $|{\mathcal{T}_{D}}^{G_{K_{D}}}| = |\mathcal{T}_{D}| = 3$ we must have that the first term of the cohomology sequence (\ref{long2}) is trivial. For the same reason, the $H^{2}$-term is also trivial and we therefore obtain the isomorphism 
$$H^{1}(G_{\mathbb{Q}},\mathcal{T}_{D}) \cong H^{1}(G_{K_{D}}, \mathcal{T}_{D}).$$ 

\begin{fact}\label{correspondence}It is shown by Satg\'e in \cite[Lemma 1.3] {Satge} that $H^{1}(G_{\mathbb{Q}},\mathcal{T}_{D})$ identifies with a subgroup of $Hom(G_{K_{D}},\mathbb{Z}/3\mathbb{Z})$. In particular, the kernel of every homomorphism $\xi$ in this subgroup fixes a cubic extension of $K_{D}$, which is non-abelian and Galois, over $\mathbb{Q}$. We add that therefore, the Galois group of this extension, over $\mathbb{Q}$, is isomorphic to $\mathbb{S}_{3}$.\end{fact}

Denote by $d(\mathcal{S}_{\lambda})$ and $d(\mathcal{S}_{\lambda'})$ the rank of the Selmer groups $\mathcal{S}_{\lambda}$ and $\mathcal{S}_{\lambda'}$, relative to the isogenies $\lambda$ and $\lambda'$, of the curves $E_{D}$ and $E_{D'}$ respectively, over $\mathbb{Q}$. Let $Cl(K_{D})_{3}$ denote the $3$-Sylow subgroup of the ideal class group $Cl(K_{D})$ of $K_{D}$. We denote by $r_{3}(D)$ the rank of the $3$-part of the ideal class group
$$r_{3}(D) \coloneqq dim_{\mathbb{F}_{3}}(Cl(K_{D})_{3}/Cl(K_{D})_{3}^{3}).$$ 
From \cite[Section 3]{Satge}, we can compute the precise rank for the Selmer groups $\mathcal{S}_{\lambda}$ and $\mathcal{S}_{\lambda'}$. 
\begin{proposition}\label{Selmer} The rank of the Selmer groups $\mathcal{S}_{\lambda}$ and $\mathcal{S}_{\lambda'}$ of the curves $E_{D}$ and $E_{D'}$ are as follows: 
$$d(\mathcal{S}_{\lambda}) = d(\mathcal{S}_{\lambda'}) = r_{3}(D), \ \ \ \ \text{if $D < -4$}$$ 
$$d(\mathcal{S}_{\lambda}) = r_{3}(D) \ \text{and} \ d(\mathcal{S}_{\lambda'}) = r_{3}(D) + 1, \ \ \ \  \text{if $D > 0$}.$$ 
\end{proposition}
\begin{proof} Our elliptic curves $E_{D}$ have $16D = 16(4m^3-27n^2)$ and with $D$ squarefree we have that $2^4 || 16D$. Therefore, Lemma 3.1 in \cite{Satge} is vacuously true. We also see that $16D \equiv (4m)^3 \equiv \pm 1~({\rm mod}~9)$, since $3 \nmid m$ by assumption. Therefore, Propositions 3.2 and 3.3 in \cite{Satge} give us the desired result. \end{proof}

\begin{remark}\label{correspondence1} Given Fact~\ref{correspondence} and Proposition~\ref{Selmer}, we see that for our family of curves $E_{D}$ (where $a = b = c = 0$ following the notation of \cite[Definition 1.12]{Satge}), each $\xi \in H^{1}(G_{\mathbb{Q}},\mathcal{T}_{D})$ belongs to the Selmer group $\mathcal{S}_{\lambda}$ of the curves $E_{D}$ over $\mathbb{Q}$, if and only if the cubic extension of $K_{D}$ fixed by the kernel of $\xi$ is unramified. Therefore, from Fact~\ref{correspondence} and Theorem~\ref{disc}, we conclude that the corresponding cubic fields are of discriminant exactly $D$. \end{remark}

Together with the discussion in Section~\ref{sec:L[mu]}, we have established the following bijections:

\begin{bijection}\label{Bij:CocyclesTriples} There is a bijection between the cocycle classes $[\xi]$ in $\mathcal{S}_{\lambda}$ and the classes of $3$-virtual units $[\mu] \in Sel_{3}(K_{D'})$ whose corresponding cubic polynomial $f_{\mu}$ defines an unramified cubic extension of $K_{D}$ and hence, $f_{\mu}$ generates a cubic field of discriminant $D$. \end{bijection}

\begin{bijection}\label{Bij:CocyclesFields}
There is a bijection between the groups $\langle [\xi] \rangle \subseteq \mathcal{S}_{\lambda}$, and the groups $\langle [\mu] \rangle \subseteq Sel_{3}(K_{D'})$ that give rise to the cubic fields $L_{[\mu]}$, unique up to conjugation, of discriminant $D$. 
\end{bijection}

\subsection{The Fundamental $3$-Descent Map}
In \cite[Section 8.4.2]{Cohen}, Cohen defines the Fundamental $3$-descent map from $E_{D}(\mathbb{Q})$ to the multiplicative group $K_{D}^{\times}/(K_{D}^{\times})^{3}$ where, by the explicit definition of the isogeny $\lambda$, each point $P = (x_{P},y_{P}) \in E_{D}(\mathbb{Q})$ is mapped to $(y_{P} + 4\sqrt{D}) \cdot {K_{D}^{\times}}^{3}$. This map is a group homomorphism with kernel equal to $\lambda'(E_{D'}(\mathbb{Q}))$. Similarly, for the dual isogeny $\lambda'$, each point $P \in E_{D'}(\mathbb{Q})$ is mapped to $(y_{P} + 12\sqrt{D}) \cdot {K_{D}^{\times}}^{3}$ and the group homomorphism from $E_{D'}(\mathbb{Q})$ to the multiplicative group $K_{D'}^{\times}/(K_{D'}^{\times})^{3}$ has kernel $\lambda(E_{D}(\mathbb{Q}))$. Given these $3$-descent maps, let us call them $\psi$ and $\psi'$, we obtain the injective group homomorphisms $\Psi$ and $\Psi'$, which we will also call \emph{Fundamental $3$-descent maps}:
\begin{equation}\label{quotient2} E_{D}(\mathbb{Q})/\lambda'(E_{D'}(\mathbb{Q})) \xrightarrow{\Psi} Ker(N_{K_{D}}: K_{D}^{\times}/ {K_{D}^{\times}}^{3} \rightarrow  \mathbb{Q}^{\times}/ {\mathbb{Q}^{\times}}^{3}),\end{equation} $$(x , y) \mapsto (y + 4\sqrt{D}) \cdot {K_{D}^{\times}}^{3}$$ 
and 
\begin{equation} E_{D'}(\mathbb{Q})/\lambda(E_{D}(\mathbb{Q})) \xrightarrow{\Psi'} Ker(N_{K_{D'}} : K_{D'}^{\times}/ {K_{D'}^{\times}}^{3} \rightarrow  \mathbb{Q}^{\times}/ {\mathbb{Q}^{\times}}^{3}),\label{quotient1} \end{equation}
$$(X , Y) \mapsto (Y + 12\sqrt{D'}) \cdot {K_{D'}^{\times}}^{3}.$$
The way these two quotient groups fit together in a short exact sequence, is shown in (\ref{splitrank}). 

We have the following, immediate corollary.

\begin{corollary}\label{cor:ranks}
The ranks of 
$$E_{D'}(\mathbb{Q}) / \lambda(E_{D}(\mathbb{Q})) \ \text{and} \  \Psi'(E_{D'}(\mathbb{Q}) / \lambda(E_{D}(\mathbb{Q}))),$$ 
as $\mathbb{F}_{3}$-vector spaces, are equal.
\end{corollary}
\begin{proof}
We observe that $\Psi'(E_{D'}(\mathbb{Q}) / \lambda(E_{D}(\mathbb{Q}))) \subseteq Sel_{3}(K_{D'})$, and the result now follows directly from the fact that the map $\Psi'$ induces an isomorphism between these two $\mathbb{F}_{3}$-vector spaces.
\end{proof}

From the definition of $D = D_{m,n}$, the curve $E_{D'}(\mathbb{Q})$ has the integral point $P_{m,n} = (12m,108n)$.

\begin{lemma}\label{Pmn}
The integral point $P_{m,n}=(12m,108n)$ belongs to 
$$P_{m,n} \in E_{D'}(\mathbb{Q}) \setminus \lambda(E_{D}(\mathbb{Q})).$$
\end{lemma}
\begin{proof}
Define the $K_{D'}$-integral element 
$$\mu_{m,n}=\frac{27n+3\sqrt{D'}}{2} \in \mathcal{O}_{D'}.$$ 
We compute its trace and norm  
$$\text{Norm}_{K_{D'}/\mathbb{Q}}(\mu_{m,n}) = (3m)^{3},$$
$$\text{Trace}_{K_{D'}/\mathbb{Q}}(\mu_{m,n}) = 27n.$$
As usual, we form the corresponding cubic polynomial
\begin{equation}\begin{gathered}\label{eqn:thePoly}f_{\mu_{m,n}}(x) = x^{3} - 3(\text{Norm}(\mu_{m,n}))^{1/3}x + \text{Trace}(\mu_{m,n}) = \\ x^{3} - 3(3m)x + 27n = x^{3}-Mx+N.\end{gathered}\end{equation}
Since $3^{2}|M$ and $3^{3}|N$, we consider the polynomial $f_{m,n}$ below, which is in standard form:
$$f_{m,n}(x) = x^{3} - mx + n.$$ 
Since $f_{m,n}$ is irreducible (by assumption, since $D \in \mathcal{D}$), this implies that $\mu_{m,n}$ is not a cube in $K_{D'}^{\times}$. This can be explicitly verified as follows: 

Assume by way of contradiction that $\mu = a^{3}, a \in K_{D'}^{\times}$. Then 
$$-\text{Trace}_{K_{D'}/\mathbb{Q}}(a) = a + \overline{a} \in \mathbb{Z} \ \text{and} \ \text{Norm}_{K_{D'}/\mathbb{Q}}(a) = (\mu\overline{\mu})^{1/3} = 3m.$$ Then one can immediately see that $-\text{Trace}_{K_{D'}/\mathbb{Q}}(a)$ is an integral root of $f_{\mu_{m,n}}(x)$ of Equation~(\ref{eqn:thePoly}), which in turn gives us that
\begin{equation}\begin{gathered}\label{eqn:1}0 = f_{\mu_{m,n}}(-\text{Trace}_{K_{D'}/\mathbb{Q}}(a)) = -(a + \overline{a})^{3}+9m(a + \overline{a})+27n \\
\iff \\
0 = \Big{(}\frac{-(a + \overline{a})}{3}\Big{)}^{3} - m\Big{(}\frac{-(a + \overline{a})}{3}\Big{)}+n = f_{m,n}\Big{(}\frac{-(a + \overline{a})}{3}\Big{)}.\end{gathered}\end{equation}
We also see that 
\begin{equation}\begin{gathered}\label{eqn:2} 27n = \mu+\overline{\mu} = a^{3}+\overline{a}^{3} = (a+\overline{a})^{3}-3a\overline{a}(a+\overline{a}) \\
\Rightarrow \frac{-(a + \overline{a})}{3} \in \mathbb{Z}.\end{gathered}\end{equation}
Equations~(\ref{eqn:1} \& \ref{eqn:2}) now imply that $f_{m,n}$ has an integral root, which is a contradiction since the corresponding discriminant $D_{m,n}$ belongs to the set $\mathcal{D}$.

Let us look now at the image of $P_{m,n}$ under the Fundamental $3$-descent map:
$$\Psi'(P_{m,n}) = 108n+12\sqrt{D'} \equiv \Big{(}\frac{1}{2}\Big{)}^{3}(108n+12\sqrt{D'}) = \mu_{m,n} \bmod (K_{D'}^{\times})^{3}.$$ 
Since, as we have just seen, $\mu_{m,n}$ is not a cube in $K_{D'}^{\times}$, this implies that $P_{m,n}$ is not in the kernel of $\Psi'$.
\end{proof}

We conclude this section with the following lemma.

\begin{lemma}\label{lem:IndependencePnts}
    Two points $P, Q \in E_{D'}(\mathbb{Q})\setminus \lambda(E_{D}(\mathbb{Q}))$ satisfy $P = kQ$, for some integer $k$, if and only if their corresponding cubic polynomials $f_{P}$ and $f_{Q}$ define the same cubic field, under conjugation. 
\end{lemma}
\begin{proof}
   We have $P, Q \in E_{D'}(\mathbb{Q})\setminus \lambda(E_{D}(\mathbb{Q}))$ with $P = kQ$, 
   
   if and only if $P \equiv rQ + \lambda(E_{D}(\mathbb{Q}))$ where $k=3k' + r$ and $r = 1,2$, 
   
   if and only if under the Fundamental $3$-descent Map $\Psi'$, the corresponding groups $\langle [\mu_{P}] \rangle$ and $\langle [\mu_{Q}] \rangle$ are equal, 
   
   if and only if $L_{[\mu_{P}]} = L_{[\mu_{Q}]}$. 
\end{proof}

\subsection{Binary Cubic Forms: Known Facts and Definitions} \label{subsec:BCF} For a more detailed exposition of the theory and facts of this section, the reader may refer to \cite{Bhar1}, \cite{Bhar2}, \cite{Cohen1} and \cite{Hambleton}.

We will always consider here the irreducible cubic polynomials 
$$f(x)=ax^{3}+bx^{2}+cx+d \in \mathbb{Z}[x],$$ whose splitting field has Galois group isomorphic to the Symmetric group $\mathbb{S}_{3}$, and whose discriminant equals the discriminant of the cubic field they define. Each such irreducible cubic polynomial, defines a class of conjugate cubic fields 
$$\{ L_{i} \}_{0\leq i \leq 2} =  \{\mathbb{Q}(\delta_{i})\}_{0\leq i\leq2}, \ \text{where} \ f(\delta_{i})=0 \ \text{for all} \ 0\leq i \leq 2,$$
and in our case, $disc(f) = disc(L_{i}), \ \text{for all} \ 0 \leq i \leq 2.$ Unless otherwise stated, the cubic field $L_{f}$ will refer to any one of the cubic fields $L_{i}$. 

By `homogenizing' $f$, we obtain the irreducible integral binary cubic form 
$$F(x,y) = ax^{3}+bx^{2}y+cxy^{2}+dy^{3},$$ which for brevity we will also denote by $$F = [a,b,c,d].$$ 
The discriminant of $F$ is given by the formula 
$$disc(F) = disc(f) = b^{2}c^{2}+18abcd-27a^{2}d^{2}-4ac^{3}-4b^{3}d.$$ 
We remind the reader that a binary cubic form is called \emph{irreducible} if and only if $F(x,1)$ is an irreducible polynomial over $\mathbb{Q}$.

Given any matrix 
$$M = \begin{pmatrix} 
	\alpha & \beta \\
	\gamma & \delta \\
	\end{pmatrix}  \in GL(2,\mathbb{Q}),$$ 
 it is straightforward to see that 
\begin{equation}\label{SimpleAction} 
disc(MF(x,y)) \coloneqq disc(F(\alpha x+\beta y,\gamma x+ \delta y)) = det(M)^{6}disc(F(x,y)). 
\end{equation} 
We deduce that the discriminant of a binary cubic form is $GL(2,\mathbb{Z})$-invariant, and therefore $SL(2,\mathbb{Z})$-invariant. 

We denote by $[F]$ the $SL(2,\mathbb{Z})$-equivalence class of the irreducible integral binary cubic forms $F$, of discriminant $disc(F)$. Under the Davenport-Heilbronn correspondence \cite{DavHeil} and \cite{DavHeil2}, every such $SL(2,\mathbb{Z})$-equivalence class $[F]$ corresponds to a triple of conjugate cubic fields $\mathbb{Q}(\delta_{i})$ of discriminant $disc(F)$ and with $F(\delta_{i},1) = f(\delta_{i}) = 0$, where $F$ is the `homogenization' of $f$; i.e. $f(x) = F(x,1)$. 

For simplicity, and where it is clear, $L$ will refer to the cubic field $L_{f} \equiv L_{F}$. Following the terminology in \cite{Hambleton}, we will call any irreducible integral binary cubic form $F$ with discriminant $disc(F) = disc(L)$, an \emph{Index Form} for $L$. We will denote the equivalence class of Index Forms of a cubic field $L$, by $IF_{L}(disc(L))$. Every $F \in IF_{L}(disc(L))$ corresponds to the maximal order $\mathcal{O}_{L}$ of $L$, under the Delone-Fadeev correspondence. 

It is well known that given any Index Form, there is a known algorithm that reduces it to the \emph{Reduced Index Form} of the same cubic field $L$. Each $SL(2,\mathbb{Z})$-equivalence class $[F]$ contains a \emph{unique} Reduced Index Form; see for example \cite[\S 3.2 \& \S3.3]{Hambleton} and \cite{Belabas}. For brevity, we will denote the Reduced Index Form of $IF_{L}(disc(L))$ by $RIF_{L}$. As we see below, the definition of the reduced form differs slightly, depending on whether the discriminant is positive or negative. Before giving the definition, we remind the reader of the Hessian $H$ and the Jacobian $J$ associated with a binary cubic form $F$.

\begin{definition}\label{defn:HessJac}
Let $F=[a,b,c,d]$ be an integral binary cubic form. The Hessian of $F$ is the binary quadratic form defined as
$$H=[A,B,C]=[b^{2}-3ac, bc-9ad, c^{2}-3bd]$$ and the Jacobian of $F$ is another binary cubic form defined as
$$J=[3aB-2Ab, 9aC-3Ac, 3bC-9Ad, 2cC-3Bd].$$
\end{definition}

\begin{definition}\label{defn:RIF} Let $F=[a,b,c,d]$ be an integral binary cubic form.
\begin{enumerate}[label=(\Alph*)]
\item If $F$ is of positive discriminant $\Delta$, the Hessian has negative discriminant equal to $B^{2}-4AC = -3\Delta$, and we say that $F$ is reduced if the following conditions hold:
\begin{enumerate}
\item $|B| \leq A \leq C$.
\item $a > 0, b \geq 0$.
\item If $b = 0$, then $d<0$.
\item If $A=B$,then $b<|3a-b|$.
\item If $A=C$,then $a<|d|$,or $a=|d|$ and $b<|c|$.
\end{enumerate}
\item If $F$ is of negative discriminant $\Delta$, let $\delta \in \mathbb{R}$ satisfy $F(\delta,1)=0$, $\rho_{1}=a\delta$ and $\rho_{2}=a\delta^{2}+b\delta$.
Let $\overline{A} = a, \overline{B} = b+\rho_{1}$, and $\overline{C} = c+\rho_{2}$. We say that $F$ is reduced if 
$$|\overline{B}| < \overline A < \overline C,$$ and either $b > 0$ or $b=0$ and $d>0$.
\end{enumerate}
\end{definition}

The Hessian $H$ and the Jacobian $J$ of a binary cubic form $F$ are known as the quadratic, and respectively the cubic, covariants of $F$ since these forms are covariant with respect to the action of $GL(2,\mathbb{Z})$. Specifically
\begin{equation}\label{reln:covariant}
    H_{MF} = MH_{F} \ \text{and} \ J_{MF} = MJ_{F}, \ \forall M \in GL(2,\mathbb{Z}),
\end{equation}
where the action of $M$ on $H$ and $J$ is defined the same way as we defined the action of $M$ on $F$ in (\ref{SimpleAction}) above.

Given an integral binary cubic form $F$ of discriminant $\Delta$, its quadratic and cubic covariants $H$ and $J$, satisfy the following equality, known as \emph{Cayley's Syzygy}:  
\begin{equation}\label{CayleySyzygy}
    J(x,y)^{2} = 4H(x,y)^{3} - 27\Delta F(x,y)^{2}.
\end{equation}
We see that if $F$ has a point $P=(x_{P},y_{P}) \in \mathbb{Q}^{2}$ such that $F(x_{P},y_{P})=1$, then this point can be easily transformed into a rational point on our elliptic curve $E_{D'}(\mathbb{Q})$. 

We will discuss Cayley's Syzygy further, later in the paper, in the context of Homogeneous Spaces for our elliptic curves, that we introduce in Section~\ref{Torsors}.

\subsection{The group of rational points $E_{D'}(\mathbb{Q})$.}
In Lemma~\ref{PointForm} below, the first statement regarding the form of every point $P \in E_{D'}(\mathbb{Q})$ is already proved in \cite[pp.71-72]{SilvTate}, for arbitrary curves of the form $y^{2}=x^{3}+ax^{2}+bx+c$. We include our proof of the first statement here, for completeness. The second statement of the lemma is specific to the curves $E_{D'}$. 

\begin{lemma}\label{PointForm}
Every point $P \in E_{D'}(\mathbb{Q})$ is of the form $P=(x/z^{2},y/z^{3})$ for some $x, y \in \mathbb{Z}^{\times}$ and $z \in \mathbb{Z}^{+}$, with $\text{gcd}(x,z) = \text{gcd}(y,z) = 1$. Furthermore, $d \coloneqq \text{gcd}(x,y) | 6$ and if $2 | d$, then $4 || y$ and $4 | x$, and if $3 | d$, then $3 || x$ and $9 | y$. 
\end{lemma}
\begin{proof} Let us write $P=(x/z_{1},y/z_{2})$, with $z_{1}, z_{2}, x, y \in \mathbb{Z}$. Without loss of generality, we may assume that $z_{1}, z_{2}$ are always positive and that $\text{gcd}(x,z_{1}) = \text{gcd}(y,z_{2}) = 1$, with $x$ and $y$ arbitrary integers. Plugging $P$ into the equation of $E_{D'}$ we obtain
\begin{equation} \label{Lemma:Points}
 x^{3}z_{2}^{2} - y^{2}z_{1}^{3} = 27\cdot16z_{1}^{3}z_{2}^{2}D.\end{equation}
 Since $\text{gcd}(x,z_{1}) = \text{gcd}(y,z_{2}) = 1$, we must have that
 $$z_{1}^{3} | z_{2}^{2}, \ \text{and} \ z_{2}^{2} | z_{1}^{3}.$$
 Since $z_{1}, z_{2}$ are both positive integers, we conclude that there is a positive integer $z$ such that 
 \begin{equation}\label{reln:xyz}z_{1} = z^{2}, \ z_{2}=z^{3}, \ \text{gcd}(xy,z) = 1.\end{equation} 
 
Equation(\ref{Lemma:Points}) now becomes
\begin{equation} \label{Lemma:Points2}
 x^{3} - y^{2} = 27\cdot16Dz^{6}.\end{equation} 
 
Let $d = \text{gcd}(x,y)$ denote the greatest common divisor of $x$ and $y$. Then, from relation(\ref{reln:xyz}), we should have that $\text{gcd}(d,z) = 1$. From equation(\ref{Lemma:Points2}) and from the fact that $D$ is squarefree and $6 \nmid D$, we conclude that $d | 6$ and it is now straightforward to see that: 
\begin{equation}\label{CasesDiv} \begin{cases}\text{if $2 | d$, then $4 || y$ and $4 | x$,} \\ 
\text{if $3 | d$, then $3 || x$ and $27 | y$.}\end{cases}\end{equation}  
\end{proof}

\begin{lemma}\label{trinomials} Every cubic polynomial $f(x) = x^{3}+ax^{2}+bx+c$, with $a, b, c \in \mathbb{Z}$, can be written in the form $F_{m,n}(X) = X^3-mX+n$, where $X= x+\frac{a}{3}$, $m=\frac{a^2}{3}-b$ and $n=c + \frac{2a^3}{27} - \frac{ab}{3}$. \end{lemma} 
\begin{proof} Straightforward. \end{proof}  

\begin{proposition}\label{prop:PntsCubicForm}
The set of points $E_{D'}(\mathbb{Q}) \setminus \lambda(E_{D}(\mathbb{Q}))$ is non-empty if and only if there exists an irreducible, integral, binary cubic form $F_{3}$ 
of discriminant $disc(F_{3}) = k^{6}D$, where $k$ is some nonzero integer, and $F_{3}$ has leading coefficient an integral cube, i.e. $F_{3}$ is of the form $F_{3} = [a^{3},b,c,d]$, for some $a, b, c, d \in \mathbb{Z}$.
\end{proposition}
\begin{proof}
$(\Rightarrow)$ Assume that $E_{D'}(\mathbb{Q}) \setminus \lambda(E_{D}(\mathbb{Q})) \neq \emptyset$. By Lemma~\ref{PointForm}, $P$ is of the form $P=(x/z^{2},y/z^{3})$. Under the Fundamental $3$-Descent map, this point corresponds to the non-trivial element 
$$\Psi'(P) = y/z^{3} + 12\sqrt{D'} \in K_{D'}^{\times}/(K_{D'}^{\times})^{3}.$$ In other words, $\Psi'(P)$ is a rational element of $K_{D'}^{\times}$ that is not a cube in $K_{D'}$.

We assume first that $\text{gcd}(6,xy)=1$ and we form the polynomial
\begin{equation}\label{FracPoly}f_{P}(X) = X^{3}-\frac{x}{12z^{2}}X+\frac{y}{108z^{3}} \in \mathbb{Q}[x].\end{equation} We see that 
its discriminant is $$disc(f_{P}) = 4(\frac{x}{12z^{2}})^{3}-27(\frac{y}{108z^{3}})^{2} = D.$$ 
By `homogenizing' $f_{P}$ we obtain the following (rational) binary cubic form, also of discriminant $D$,
\begin{equation}\begin{gathered}\label{FormDisc} F_{P}(X,Y) = X^{3}-\frac{x}{12z^{2}}XY^{2}+\frac{y}{108z^{3}}Y^{3} = \\ X^{3}-3xX(\frac{Y}{6z})^{2}+2y(\frac{Y}{6z})^{3} =
F_{3}(X,Y/6z) = MF_{3}(u,v),
\end{gathered}\end{equation}
where 
$$M=\begin{pmatrix} 1& 0 \\ 0 & 1/(6z) \end{pmatrix} \in GL(2,\mathbb{Q}), \ \text{and}$$ 
$$F_{3}(u,v) = u^{3}-3xuv^{2}+2yv^{3}$$ 
is an integral binary cubic form with leading coefficient an integral cube. Furthermore, from (\ref{FormDisc}) we deduce that 
$$disc(F_{3}) = (det(M)^{6})^{-1}disc(F_{P}) = (6z)^{6}D.$$
It remains to show that $F_{3}$ is irreducible. We observe that the element 
$$\mu_{P} \coloneqq y +12z^{3}\sqrt{D'} \equiv y/z^{3} + 12\sqrt{D'} = \Psi'(P) \bmod (K_{D'}^{\times})^{3},$$ is not a cube in $K_{D'}$ and therefore neither is 
$\nu_{P}=-y+12z^{3}\sqrt{D'} = -\overline{\mu_{P}}$, where the bar denotes the Galois conjugate of an element. By employing Cardano's classical formula for the zero's of cubic polynomials, \cite[\S 1.4, pg.25]{Hambleton}, we verify via direct computation that the zero's of $F_{3}(u,1)$ are given by
$$\delta_{P}^{(i)} = w^{i}(\nu_{P})^{(1/3)}+w^{-i}(\overline{\nu_{P}})^{(1/3)},$$ 
$$\text{for} \  0\leq i \leq2, \ \text{and} \ w \ \text{a cube root of unity}.$$ 
With $\nu_{P}$ not being a cube in $K_{D'}^{\times}$, this proves that $F_{3}(u,1)$ cannot have a quadratic factor, and consequently $F_{3}(u,1)$, and therefore $F_{3}(u,v)$, are irreducible.

Finally, if either $2$ or $3$ divide $x$ or $y$, then from Lemma~\ref{PointForm}(\ref{CasesDiv}) we know exactly the powers of $2$ or $3$ that divide $x$ and $y$. 
Therefore, the general form of the binary cubic form $F_{3}$ and the corresponding transformation matrix $M$, are:
$$F_{3} = [1,0,-3x/(9^{i}4^{j}),2y/(27^{i}4^{j})],$$ 
$$\text{with} \ disc(F_{3}) = (3^{(1-i)}2^{(1-j)}z)^{6}D,$$ 
$$\text{and both} \ 3x/(9^{i}4^{j}) \ \text{and} \  2y/(27^{i}4^{j}) \ \text{are integers, by Lemma~\ref{PointForm}(\ref{CasesDiv})},$$
$$\text{and} \ M=\begin{pmatrix} 1 & 0 \\ 0 & 1/(3^{(1-i)}2^{(1-j)}z)\end{pmatrix} \in GL(2,\mathbb{Q}),$$  $$\text{with} \ i,j \in \{0,1\} \ \text{and}$$
$$i = 1 \ \text{if and only if} \ 3 | x,$$
$$j = 1 \ \text{if and only if} \ 2 | x.$$

The elements $\mu_{P}$ will of course change accordingly and we observe that in all cases, $\mu_{P}$ is an integral primitive $3$-virtual unit and equivalently, $F_{3}(u,1)$ is always in standard form.

$(\Leftarrow)$ Assume now that there exists an integral, irreducible,  binary cubic form $F_{3} = [a^{3},b,c,d]$, of discriminant equal to $disc(F_{3}) = k^{6}D$, for some nonzero integer $C$. We observe that
\begin{equation}
    \begin{gathered}
        F_{3}(u,v)=a^{3}u^{3}+bu^{2}v+cuv^{2}+dv^{3} = \\
        (au)^{3}+b(au)^{2}(\frac{v}{a^{2}})+ca^{3}(zu)(\frac{v}{a^{2}})^{2}+da^{6}(\frac{v}{a^{2}})^{3} = \\
        F(au,\frac{v}{a^{2}}) = MF(u,v),
    \end{gathered}
\end{equation}
where $$F(u,v) = u^{3}+bu^{2}v + ca^{3}uv^{2} + da^{6}v^{3},$$
$$\text{and} \ M=\begin{pmatrix} a & 0 \\ 0 & 1/a^{2}\end{pmatrix} \in GL(2,\mathbb{Q}).$$
The discriminant of the new, also irreducible, integral binary cubic form is equal to $disc(F) = det(M)^{-1}disc(F_{3}) = (ak)^{6}D$. Consider the polynomial $F(u,1) = u^{3}+bu^{2} + ca^{3}u + da^{6}$. By employing Lemma~\ref{trinomials}, we obtain the irreducible polynomial 
$$F_{M,N}(U) = U^{3} - MU + N,$$
$$\text{with} \ M = \frac{b^{2}-3ca^{3}}{3}, \ N = \frac{27da^{6}+2b^{3}-9bca^{3}}{27}, \ \text{and}$$ \begin{equation}\label{GetPnt}disc(F_{M,N}) = 4M^{3}-27N^{2} = (ak)^{6}D.\end{equation}
But now, relation(\ref{GetPnt}), together with the fact that $F_{M,N}$ is irreducible, proves that the point $P_{M,N} = (\frac{12M}{(ak)^{2}},\frac{108N}{(ak)^{3}})$ belongs to $E_{D'}(\mathbb{Q})\setminus \lambda(E_{D}(\mathbb{Q}))$, via the same proof of Lemma~\ref{Pmn}.  
\end{proof}

Proposition~\ref{prop:PntsCubicForm} establishes a correspondence between the irreducible integral binary cubic forms of discriminant equal to $k^{6}D$, for some nonzero integer $k$, and the points of $E_{D'}(\mathbb{Q})\setminus \lambda(E_{D}(\mathbb{Q}))$. Under the Delone-Fadeev correspondence \cite{DelFad}, each such irreducible integral binary cubic form corresponds to an order $\mathcal{O}_{P} \subseteq \mathcal{O}_{L}$ of conductor $[\mathcal{O}_{L} : \mathcal{O}_{P}]$ equal to an integral cube, where $L$, as in the introduction, is a cubic field of discriminant $D$ and whose Galois closure has Galois group isomorphic to $S_{3}$. We have therefore established the following:
\begin{bijection}\label{bij:CubeConductor} There exists a bijection between the points 
$$P \in E_{D'}(\mathbb{Q})\setminus \lambda(E_{D}(\mathbb{Q}))$$ 
and suborders 
$$\mathcal{O}_{P} \subseteq \mathcal{O}_{L}, \ \text{with conductor a cube,}$$ where $L$ is as above.
\end{bijection}

\subsection{From IBCF's to Index Forms}
In this section we will show that such an order $\mathcal{O}_{P} \subseteq \mathcal{O}_{L}$ exists, if and only if there exists an Index Form in $IF_{L}(D)$ with leading coefficient a cube. Under the Delone-Fadeev correspondence, this Index Form corresponds to the maximal order $\mathcal{O}_{L}$. 

Note that Proposition~\ref{prop:PntsCubicForm} only implies the existence of such a form of discriminant $k^{6}D$ for some nonzero integer $k$. 

We will first need the following lemma:
\begin{lemma}\label{lem:basis}
To each point $P= (X_{P}/Z^{2},Y_{P}/Z^{3}) \in E_{D'}(\mathbb{Q})\setminus \lambda(E_{D}(\mathbb{Q}))$ there corresponds an order $\mathcal{O}_{P} \subseteq \mathcal{O}_{L}$ of conductor equal to $k^{3}$, for some non-zero integer $k$, and with basis $\mathcal{B}_{P} = \{1, \delta, \delta^{2}-A\}$, where $\delta$ is a root of the monic irreducible polynomial 
$$f_{P}(x)=x^{3}-Ax+B \coloneqq x^{3}-3X_{P}/(9^{i}4^{j})x+2Y_{p}/(27^{i}4^{j}) \in \mathbb{Z}[x].$$ The cubic field $L$ corresponds to the one under Bijection~(\ref{bij:CubeConductor}).
\end{lemma}
\begin{proof}
  To each point $P = (X_{P}/Z^{2},Y_{P}/Z^{3}) \in E_{D'}(\mathbb{Q})\setminus \lambda(E_{D}(\mathbb{Q}))$, Proposition~\ref{prop:PntsCubicForm} gives us the irreducible integral monic binary cubic form $$F_{P} = [1,0,-3X_{P}/(9^{i}4^{j}),2Y_{P}/(27^{i}4^{j})],$$ 
$$\text{with} \ disc(F_{P}) = (3^{(1-i)}2^{(1-j)}Z)^{6}D,$$
which under the Delone-Fadeev Correspondence, it corresponds to an order $\mathcal{O}_{P}$ of conductor equal to the integral cube $(3^{(1-i)}2^{(1-j)}Z)^{3}$.
To ease notation, we will denote this form by 
$$F_{P} \eqqcolon [1,0,-A,B].$$
We denote by $\delta$ any one of the roots of the monic irreducible polynomial 
    $$F_{P}(x,1) \coloneqq f_{P}(x)=x^{3}-Ax+B,$$ of discriminant $4A^{3}-27B^{2} = k^{6}D$, and we let 
    $$w_{1}=\delta \ \text{and} \ w_{2}=\delta^{2}-A.$$ 
To compute a basis of the corresponding order $\mathcal{O}_{P}$, 
    we first observe that 
    \begin{equation}\label{rln:basis}
    \begin{array}{l}
    w_{1}^{2} = w_{2}+A \\
    w_{1}w_{2} = -B \\
    w_{2}^{2} = -Bw_{1} -Aw_{2}.
    \end{array}
    \end{equation}
    Relations~(\ref{rln:basis}) imply that indeed the set 
    $$\mathcal{B}_{P} = \{1,w_{1},w_{2}\}$$
    is a basis for the order $\mathcal{O}_{P}$ \cite[Theorem 9]{Bhar1}. This basis is called \emph{normal}.
    \end{proof}
    
In the next proposition, we will employ Voronoi's algorithm (see for example \cite[Theorem 1.6]{Hambleton} or \cite[Theorems 2 \& 3]{AlWil}), and we will compute from $\mathcal{B}_{P}$, an explicit basis $\mathcal{B}_{L}$ for the maximal order $\mathcal{O}_{L}$ of $L$. We will then show how the corresponding Index Form must have leading coefficient an integral cube. 

\begin{proposition}\label{prop:CubeCoeff}
    The set $E_{D'}(\mathbb{Q})\setminus \lambda(E_{D}(\mathbb{Q}))$ is not empty, if and only if there exists an Index Form $\in IF_{L}(D)$ with leading coefficient a cube, where the cubic field $L$ corresponds to the one under Bijection~(\ref{bij:CubeConductor}).
\end{proposition}
\begin{proof}
    $(\Leftarrow)$ Denote by $F=[a^{3},b,c,d] \in IF_{L}(D)$ the given Index Form that corresponds to one of our cubic fields $L$. Then $F(1/a,0)=1$ and therefore we have the point 
    $$P = (4J(1/a,0),4H(1/a,0)) \in E_{D'}(\mathbb{Q})\setminus \lambda(E_{D}(\mathbb{Q})),$$ 
    via Cayley's Syzygy~(\ref{CayleySyzygy}) and the fact that $F$ is irreducible.
    
    $(\Rightarrow)$ Let $P=(X_{P}/Z,Y_{P}/Z) \in E_{D'}(\mathbb{Q})\setminus \lambda(E_{D}(\mathbb{Q}))$ and let 
    $$F_{P} = [1,0,-A,B] \in \mathbb{Z}[x,y]$$
    be the corresponding irreducible integral binary cubic form, as defined in Lemma~\ref{lem:basis}, 
with corresponding irreducible polynomial 
    $$f_{P}(x)=x^{3}-Ax+B,$$
    and $\delta$ being any one of its roots. We know from Proposition~\ref{prop:PntsCubicForm} and \cite[\S 1.4, pg.20]{Hambleton} that
    $$disc(F_{P}) = (3^{(1-i)}2^{(1-j)}Z)^{6}D = D[\mathcal{O}_{L} : \mathbb{Z}[\delta]]^{2}.$$
    The index of $\mathbb{Z}[\delta]$ in the maximal order is therefore equal to 
    $$i(\delta) = [\mathcal{O}_{L} : \mathbb{Z}[\delta]] = (3^{(1-i)}2^{(1-j)}Z)^{3}.$$ Since $D = 4m^{3}-27n^{2} \equiv 1 \bmod 4$ is squarefree and $3 \nmid D$, straightforward computations by following Voronoi's algorithm via the proof of \cite{AlWil}, will give us an integral basis for $\mathcal{O}_{L}$ via the normal basis of $\mathcal{O}_{P}$ established in Lemma~\ref{lem:basis}. As we see below, the basis for $\mathcal{O}_{L}$ differs, depending on whether $3|A$ or not. We recall from Proposition~\ref{prop:PntsCubicForm} that $3 | A$ if and only if $3 \nmid X_{P}$. Furthermore, for each prime $p$, we easily see that the exponent $r_{p}$ defined in \cite[pg. 164]{AlWil} and which decides the exponent of $k$ below, in our case is always $0$ or $3$. More specifically, 
   \begin{equation} 
   \begin{cases}
      r_{2} \ \text{and} \ r_{3} \ \text{are either} \ 0 \ \text{or} \ 3, \\
      r_{p} = 3 \ \text{for} \ p|Z \ \text{and} \ p|D   \\
      r_{p} = 3 \ \text{for} \ p|Z \ \text{and} \ p\nmid D  \\
      r_{p} = 0 \ \text{in all other cases.}
    \end{cases} 
    \end{equation}
    The basis $\mathcal{B}_{L}$ is as follows:
\begin{enumerate}
    \item $3|A$. In this case, \cite[Theorem 2]{AlWil} gives us the following basis for $\mathcal{O}_{L}$:
    $$\mathcal{B}_{L} = \Big{\{} 1, \delta, \frac{t^{2}-A+t\delta+\delta^{2}}{k^{3}}\Big{\}},$$
    where $k = 3\cdot2^{(1-j)}Z$, and $j = 1$ if and only if $2 | X_{P}$. 
    \item $3 \nmid A$. In this case, \cite[Theorem 3]{AlWil} gives us the following basis:
    $$\mathcal{B}_{L} = \Big{\{} 1, \frac{-t + \delta}{3}, \frac{t^{2}-A+t\delta+\delta^{2}}{9k^{3}}\Big{\}},$$
    where $k = 2^{(1-j)}Z$, and $j = 1$ if and only if $2 | X_{P}$. 
\end{enumerate}
We recall that 
$$\mathcal{B}_{P} = \{ 1, \delta, \delta^{2} - A\} = \{ 1, w_{1}, w_{2}\}.$$ 
We let 
$$\mathcal{B}_{L} = \{1, u_{1}, u_{2}\}$$ 
and we see that in the first case above, 
$$w_{1} = \delta = u_{1}, \ \text{and}$$
$$w_{2} = k^{3}u_{2}-tu_{1}-t^{2},$$
which gives us the following matrix
\begin{equation}\mathcal{U}_{1} = \begin{pmatrix} 1 & 0 & -t^{2} \\ 0 & 1 & -t \\ 0 & 0 & k^{3}
\end{pmatrix} \in GL(3,\mathbb{Q}),
\end{equation}
such that 
$$\begin{pmatrix} 1 & u_{1} & u_{2} \end{pmatrix} \mathcal{U}_{1} = \begin{pmatrix} 1 & w_{1} & w_{2}\end{pmatrix}.$$
In the second case, we have that
$$w_{1} = 3u_{1}+t, \ \text{and}$$
$$w_{2} = -3tu_{1}+9k^{3}u_{2}-2t^{2}.$$
The corresponding matrix in this case is
\begin{equation}\mathcal{U}_{2} = \begin{pmatrix} 1 & t & -2t^{2} \\ 0 & 3 & -3t \\ 0 & 0 & 9k^{3}
\end{pmatrix} \in GL(3,\mathbb{Q}).
\end{equation}

From the matrices $\mathcal{U}_{1}, \mathcal{U}_{2} \in GL(3,\mathbb{Q})$ we can extract the $2\times2$ transformation matrices, $M_{1}$ and $M_{2}$, which transform, in either case, the irreducible integral binary cubic form   $F_{P}$ to the Index Form that carresponds to the basis $\mathcal{B}_{L}$. These matrices are: 
\begin{equation} M_{1} = \begin{pmatrix} 1 & -t  \\ 0 & k^{3}
\end{pmatrix},  M_{2} = \begin{pmatrix} 3 & -3t  \\ 0 & 9k^{3}
\end{pmatrix} \in GL(2,\mathbb{Q}).
\end{equation}
Given these matrices $M_{1}, M_{2}$, from the form $F_{P}=[1,0,-A,B]$ and by following \cite[\S 3.4, Corollary 3.4 \& Remark 3.2]{Hambleton}, we obtain the following Index Form for each case:
\begin{equation} 
F_{1} = [k^{3},3t,\frac{3t^{2}-A}{k^{3}},\frac{t^{3}-At+B}{k^{6}}], 
\ \
F_{2}=[k^{3},t,\frac{3t^{2}-A}{9k^{3}},\frac{t^{3}-At+B}{27k^{6}}].
\end{equation}
Direct computation of the discriminant gives us in each case 
$$
disc(F_{1}) = \frac{4A^{3}-27B^{2}}{k^{6}} = D, \
disc(F_{2}) = \frac{4A^{3}-27B^{2}}{(3k)^{6}} = D,$$
verifying that indeed, $F_{1}$ and $F_{2}$ are Index Forms and have leading coefficient a rational cube.
\end{proof}

\section{On the rank of $E_{D}/\mathbb{Q}$}\label{sec:rank}
\subsection{On the parity of the rank}
The following commutative diagram contains all the important groups related to the elliptic curves $E_{D}$. Since the curves $E_{D}$ and $E_{D'}$ do not have any non-trivial rational $3$-torsion points, the exact sequences below are a more simplified version of the standard sequences that emerge from the method of \emph{descent by a rational $3$-isogeny}, as found for example in \cite{Kloo}, \cite{Silverman} and \cite{Top}.
\small
\begin{center}
\begin{equation}\label{diag:sequences}
\begin{tikzcd}[column sep=small]
	& 0 \arrow[d] & 0 \arrow[d] &  & \\
        0 \arrow[r] &  E_{D'}(\mathbb{Q})/\lambda(E_{D}(\mathbb{Q})) \arrow[r] \arrow[d]  & \mathcal{S}_{\lambda}(E_{D}) \arrow[r]  \arrow[d] & \Sha(E_{D})[\lambda]  \arrow[r] \arrow[d]  & 0 \\  
        0 \arrow[r] &  E_{D}(\mathbb{Q})/3(E_{D}(\mathbb{Q})) \arrow[r] \arrow[d]  & \mathcal{S}_{3}(E_{D}) \arrow[r]  \arrow[d] & \Sha(E_{D})[3]  \arrow[r] \arrow[d]  & 0 \\
        0 \arrow[r] &  E_{D}(\mathbb{Q})/\lambda'(E_{D'}(\mathbb{Q})) \arrow[r]  \arrow[d] & \mathcal{S}_{\lambda'}(E_{D'}) \arrow[r] \arrow[d]  & \Sha(E_{D'})[\lambda']  \arrow[r]  & 0 \\
       & 0 &  \frac{\Sha(E_{D'})[\lambda']}{\lambda(\Sha(E_{D})[3])} \arrow[d]& & \\
       & & 0 & &
        \end{tikzcd}
        \end{equation}
\end{center}    
\normalsize

Given the non-degenerate alternating pairing defined by Cassels in \cite{Cassels}, by assuming finiteness of the $3$-primary part $\Sha(E_{D})[3^{\infty}]$, the torsion group $\Sha(E_{D})[3]$ becomes even dimensional. Furthermore, the term
$\frac{\Sha(E_{D'})[\lambda']}{\lambda(\Sha(E_{D})[3])}$ is also even-dimensional as an $\mathbb{F}_{3}$-vector space \cite[Prop. 4.2]{Bhar2}. Therefore, from the diagram above we conclude:
\begin{equation}\label{parity} 
rank(E_{D'}(\mathbb{Q})) \equiv d(\mathcal{S}_{3}(E_{D})) \equiv d(\mathcal{S}_{\lambda}(E_{D})) + d(\mathcal{S}_{\lambda'}(E_{D'})) \bmod 2. \end{equation}

\begin{corollary}\label{ParityRanksDiagram}
    By assuming finiteness of the $3$-primary part $\Sha(E_{D})[3^{\infty}]$ of the Tate-Shafarevich group $\Sha(E_{D})$, the rank of the elliptic curves $E_{D}$, and consequently of $E_{D'}$, is always even for negative $D \in \mathcal{D}$ and is always odd for positive $D \in \mathcal{D}$.
\end{corollary}
\begin{proof}
    The result follows directly from Proposition~\ref{Selmer}, Diagaram~(\ref{diag:sequences}), and relation~(\ref{parity}).
\end{proof}

\subsection{Lower bounds}\label{subsec:LowerBounds}
As we have already mentioned in Section~\ref{sec:L[mu]}, for fundamental discriminants such as $D$, Hasse gives the number of non-conjugate cubic fields of discriminant $D$ to be equal to $\frac{(3^{r_{3}(D)}-1)}{2}$. Given $D \in \mathcal{D}$, let us denote by $\#(cubic)$ the number of non-conjugate cubic fields $L_{[\mu]}$ of discriminant $D$ whose $SL(2,\mathbb{Z})$-equivalence class of Index Forms contains an Index Form with leading coefficient a rational cube. Let $$\mathcal{M} = \langle [\mu]_{i} \ | \ 1 \leq i \leq \#(cubic) \rangle $$ denote the subgroup of $Sel_{3}(K_{D'})$ generated by the corresponding classes of $3$-virtual units.  We denote by $\mathfrak{r}_{3}$ its rank as an $\mathbb{F}_{3}$-vector space. Then the rank of $E_{D}$, and therefore of $E_{D'}$, is bounded as follows:
\begin{proposition}\label{Rank} By assuming finiteness of the $3$-primary part $\Sha(E_{D})[3^{\infty}]$ of the Tate-Shafarevich group $\Sha(E_{D})$ of the elliptic curves $E_{D}$, the rank of the elliptic curves $E_{D}$, and therefore of $E_{D'}$, is bounded as follows: \\
(1) If $\mathfrak{r}_{3}$ is odd, then
$$(i) \  \ 2 \leq  \mathfrak{r}_{3} + 1 \leq rank(E_{D}) \leq 2r_{3}(D), \ \text{if} \ D < -4$$  
$$(ii) \ \ 1 \leq \mathfrak{r}_{3} \leq rank(E_{D}) \leq 2r_{3}(D) +1, \ \text{if} \ D > 4.$$
(2) If $\mathfrak{r}_{3}$ is even, then
$$(i) \  \ 2 \leq  \mathfrak{r}_{3} \leq rank(E_{D}) \leq 2r_{3}(D), \ \text{if} \ D < -4$$  
$$(ii) \ \ \mathfrak{r}_{3} +1 \leq rank(E_{D}) \leq 2r_{3}(D) +1, \ \text{if} \ D > 4.$$
 \end{proposition}
\begin{proof} From the definition of $\mathcal{M}$, and via the results of Proposition~\ref{prop:CubeCoeff}, every element of $\mathcal{M}$ corresponds, under the Fundamental $3$-descent map $\Psi'$, to a point in $E_{D'}(\mathbb{Q}) \setminus \lambda(E_{D}(\mathbb{Q}))$. Therefore, from Corollary~\ref{cor:ranks}, Lemma~\ref{Pmn} and the short exact sequence~(\ref{splitrank}), we have that 
$$1 \leq \mathfrak{r}_{3} = dim_{\mathbb{F}_{3}}(E_{D'}(\mathbb{Q})/\lambda(E_{D}(\mathbb{Q}))) \leq rank(E_{D}(\mathbb{Q}).$$ This, together with the parity result of Corollary~\ref{ParityRanksDiagram}, establish the lower bounds.

We now recall that $rank(E_{D}(\mathbb{Q})) \leq dim_{\mathbb{F}_{3}}(\mathcal{S}_{3})$ and therefore, from the second vertical exact sequence of Diagram~(\ref{diag:sequences}) and from Proposition~\ref{Selmer}, we obtain the upper bounds, again together with the parity result of Corollary~\ref{ParityRanksDiagram}.
\end{proof}

An immediate corollary of Proposition~\ref{Rank} is the following: 
\begin{corollary}\label{sha} For the negative discriminants $D \in \mathcal{D}$ and their corresponding imaginary quadratic fields $K_{D}$, if $r_{3}(D) = 1$ then, by assuming finiteness of $\Sha(E_{D})[3^{\infty}]$, we always have that $\text{dim}_{\mathbb{F}_{3}}(\Sha(E_{D}))[3] = 0.$
\hfill $\square$ \end{corollary}

\section{Homogeneous Spaces}\label{Torsors}
We recall the definition of a cover map between two curves (see for example \cite[Remark 24]{Bhar2}), adjusted to our notation and context.
\begin{definition} 
A $\lambda$-covering for our elliptic curve $E_{D'}/\mathbb{Q}$, is a rational map between curves $C \rightarrow E_{D'}$ which is a twist of $\lambda$. By descent, the group $H^{1}(G_{\mathbb{Q}},\mathcal{T}_{D})$ is in bijection with isomorphism classes of such $\lambda$-coverings. \end{definition}

We saw earlier in the paper, via Cayley's Syzygy~(\ref{CayleySyzygy}), that there is an explicit degree-$3$ map $\varphi$ from the genus-$1$ curves
$$C: F(x,y)=z^{3}, \ F \in IF_{L}(D)$$ to the elliptic curve $E_{D'}$, as follows:

\begin{align*}
\varphi: C & \longrightarrow E_{D'} \\
(x:y:z) & \longmapsto (4J(x/z,y/z),4H(x/z,y/z)).
\end{align*}
These curves $C/\mathbb{Q}$ are called Homogeneous Spaces or Torsors for our elliptic curve $E_{D'}/\mathbb{Q}$.

Given Bijection~\ref{Bij:CocyclesFields}, and under the Davenport-Heilbronn correspondence, we obtain the following bijection:
\begin{bijection}\label{Bij:Torsors}
Each $SL(2,\mathbb{Z})$-equivalence class of Index Forms $[F]$ of discriminant $D \in \mathcal{D}$, gives an equivalence class $[C]$ of Torsors for $E_{D'}$, to which there corresponds a subgroup $\langle [\xi] \rangle \subseteq \mathcal{S}_{\lambda}$ that fixes an unramified cubic extension $L_{[\mu]}$ of $K_{D}$, where $\langle [\xi] \rangle$ corresponds to $\langle [\mu] \rangle$ under bijections~\ref{Bij:CocyclesTriples}~and~\ref{Bij:CocyclesFields}.
\end{bijection}

\begin{proposition}\label{ViolateHassePrinciple}
A cocycle $\xi \in \mathcal{S}_{\lambda}$ corresponds to the trivial class in $\Sha(E_{D})[\lambda]$ if and only if the group $\langle [\xi] \rangle$ it generates, corresponds to a class of Index Forms $IF_{L}(D)$ that contains a form with leading coefficient an integral cube.
\end{proposition}
\begin{proof}
A cocycle $\xi \in \mathcal{S}_{\lambda}$ corresponds to the trivial class in $\Sha(E_{D})[\lambda]$, 

if and only if it belongs to the image of the connecting homomorphism $\delta$ of the short exact sequence~(\ref{selsha}), 

if and only if there is a rational point $P \in E_{D'}(\mathbb{Q})\setminus\lambda(E_{D}(\mathbb{Q}))$ such that $\delta(P) = \xi$, 

if and only if by Proposition~\ref{prop:CubeCoeff} and under Bijection~\ref{Bij:Torsors} there is a corresponding Index Form in $IF_{L}(D)$ with leading coefficient an integral cube.
\end{proof}

\subsection{A Remark regarding Solvability and Uniform Boundedness}\label{subs:solvability}
We call a cocycle $\xi \in \mathcal{S}_{\lambda}$ \emph{solvable} if it is in the image of the connecting homomorphism $\delta$ of (\ref{selsha}). From Proposition~\ref{ViolateHassePrinciple}, $\xi$ is solvable if and only if any (and therefore every) Torsor in the corresponding class $[C]$,  under Bijection~(\ref{Bij:Torsors}), has a rational point. We will also call such a class of Torsors,  \emph{solvable}. Let $C_{1}: F_{1}(x,y)=z^{3}$ and $C_{2}: F_{2}(x,y)=z^{3}$ be two Torsors that belong to a solvable class $[C]$, and let $Q_{1}=(x_{1}:y_{1}:z)$ be a rational point on $C_{1}$. The transformation matrix in $SL(2,\mathbb{Z})$ that transforms $F_{1}$ into $F_{2}$, also transforms $Q_{1}$ into a rational point $Q_{2}$ of $C_{2}$, but we notice that it leaves the $z$-coordinate unchanged. Therefore, $Q_{2}$ is of the form $Q_{2}=(x_{2}:y_{2}:z)$. Furthermore, since by Proposition~\ref{ViolateHassePrinciple} there is always an Index Form of the form 
$$F_{a}=[a^{3},b,c,d], \ \text{for some} \ a,b,c,d \in \mathbb{Z},$$ the point $Q_{a} = (1, 0 , a)$ is always a point that corresponds to a solvable class of Torsors, and every other Torsor in $[C]$ must have a point of the form $Q = (x:y:a)$. More specifically, if $C: F(x,y)=z^{3} \in [C]$ is such that $F(x,y) = MF_{a}(x,y)$, for some $M \in SL(2,\mathbb{Z})$, then the $x$- and $y$- coordinates of the point $Q$ on $C$, are determined by $M$ as follows:
$$\begin{pmatrix}
    x \\ y
\end{pmatrix} = M^{-1}\begin{pmatrix}
    1 \\ 0
\end{pmatrix}.$$

The rank of the solvable part of $\mathcal{S}_{\lambda}$ is equal to $\mathfrak{r}_{3}$, which we defined in Section~\ref{subsec:LowerBounds}. The non-conjugate \emph{monogenic} cubic fields of discriminant $D$ form a subgroup of $\mathcal{M}$ of rank $\leq \mathfrak{r}_{3}$. We recall that we always have $1 \leq \mathfrak{r}_{3}$ since we always have at least one monogenic cubic field of discriminant $D$, namely the one defined by $f_{m,n}$. To determine which of the $\Big{(}\frac{3^{r_{3}(D)}-1}{2}\Big{)}$-many non-conjugate cubic fields $L$ of discriminant $D$ are monogenic, is enough to compute their $RIF_{L}$ and observe whether it has leading coefficient $1$. The $RIF_{L}$ has rather small coefficients bounded by the discriminant, and there are known algorithms to compute both the cubic field $L$ and their corresponding $RIF_{L}$; see for example \cite{Belabas} and \cite{Cremona} respectively. 

To compute the whole rank $\mathfrak{r}_{3}$ now, by the discussion in the first paragraph of this section, amounts to deciding whether the Thue equations $F(x,y) = \pm 1$ have \emph{rational} solutions, say of the form $(x/a,y/a)$ for any nonzero integer $a$, where $F \in IF_{L}(D)$ is any representative. The problem of determining only the \emph{integral} solutions can be solved efficiently \cite[pp.1-2]{Evertse}. When looking for rational solutions, this problem is equivalent, as we can see from Lemma~\ref{PointForm} and Cayley's Syzyzygy~(\ref{CayleySyzygy}), to finding \emph{integral} points of infinite order on the elliptic curves $E_{a}: y^{2}=x^{3}-27 \cdot 16Da^{6}$. These curves $E_{a}$ are all isomorphic to each other and to $E_{D'}$. 

For any irreducible integral binary cubic form $f(x,y)$, denote by $N_{f}(m)$ the number of \emph{integral} solutions of the Thue equation $C_{f,m}: f(x,y)=m$. Silverman in \cite[Theorem A]{Silverman2} gave an upper bound for $N_{f}(m)$ in terms of the rank of the elliptic curve $E_{f,m}: y^{2} = x^{3} + m^{2}disc(f)$, $E_{f,m}$ being the Jacobian of $C_{f,m}$. In the same paper, Silverman asked the question \cite[pg.396]{Silverman2}:
\begin{equation}\label{question}\text{Does} \ \sup_{m \ cube-free}N_{f}(m) = \infty \ ?\end{equation}
Obviously if $N_{f}(m)$ could get arbitrarily large this would imply that the rank of the corresponding elliptic curves $E_{f,m}$ would also get arbitrarily large, providing at the same time a negative answer to the open question of whether the ranks of elliptic curves over $\mathbb{Q}$ are uniformly bounded. In Silverman's case however, $N_{f}(m)$ can be large only very rarely. But he only considered the case where $m$ is \emph{cube-free}, which is a case that never occurs for our elliptic curves $E_{D'}$. We remark for clarity that the elliptic curves $E_{D'}$ are the Jacobians of the curves $F(x,y) = a^{3}, F \in IF_{L}(D)$. Specifically for the case of $m$ being non-cube-free, Silverman cites Mahler's result \cite{Mahler}, which gives that in this case there are infinitely many $m$ for which $N_{f}(m) > (\log|m|)^{1/4}$. 

It would be very interesting, within the context of Uniform Boundedness, to ask Silverman's question~(\ref{question}) again, specifically for the curves $E_{D'}$ and with the condition that $m$ is a cube. 

In the same spirit, it also makes sense to consider infinite families of quadratic number fields with ideal class group of large $3$-rank, and with a large number of corresponding cubic monogenic fields. To the best of our knowledge, the authors of \cite{Levin} give the current highest lower bound for the $3$-rank of such families, which is $5$. Moreover, the authors of \cite{DavSpYoo} construct an infinite family of triples of polynomials of the same discriminant, defining distinct monogenic cubic fields. Even though their discriminant is never squarefree, it does raise the question whether similar constructions could yield infinite families of such tuples of polynomials, which have the same squarefree discriminant. 

Along these lines let us add here that the three famous elliptic curves of Quer \cite{Quer} of rank $12$, are elliptic curves of the form $E_{D}$ with $D = D_{m,n}$ and 
$(m_{1} , n_{1}) = (-363179 , 89599201), (m_{2} , n_{2}) = (-669454 , 264029105)$ and $(m_{3} , n_{3}) = (-502270 , 313570727)$. The $3$-rank of the ideal class group of the corresponding imaginary quadratic number fields $K_{D}$ is exactly $6$, implying that in this case the whole part of the corresponding Selmer group $\mathcal{S}_{\lambda}$ is solvable.  

We close this section with a lemma that gives us some information on whether the point $Q_{a}=(1:0:a)$ is a point on the Reduced Index Form, in the case of a solvable Torsor.

We denote by $C_{RIF_{L}} \in [C]$ the Torsor whose corresponding binary cubic form is the $RIF_{L}$ of $IF_{L}(D)$. In the case of positive discriminants $D>0$, and by recalling from Section~\ref{subsec:BCF} that the Reduced Index Form is unique in its equivalence class, we can say the following regarding the possible existence of a rational point on the corresponding Torsor.
\begin{corollary} Assume $F$ is an Index Form of discriminant $0 < D \in \mathcal{D}$, whose corresponding Torsor $C:F(x,y)=z^{3}$ defines a solvable class. Then the Reduced Index Form $RIF_{L}$, with $C_{RIF_{L}} \in [C]$, is either of the form $RIF_{L}=[a^{3},b,c,d]$, and therefore the point $Q_{a} = (1,0,a)$ is a point on the corresponding Torsor $C_{RIF_{L}}$, or every rational point $Q =(x_{a}:y_{a}:a)$ of $C_{RIF_{L}}$, evaluated at the Hessian $H$ of $RIF_{L}$, should satisfy the inequality 
$$H(x_{a}/a,y_{a}/a) \geq \frac{1}{2a^{2}}\sqrt{3D},$$ equivalently $$H(x_{a},y_{a}) \geq \frac{1}{2}\sqrt{3D}.$$    \end{corollary}
\begin{proof}
    The proof is almost identical to that of \cite[Lemma 5.1]{Bennett} and so we omit it.
\end{proof}

\section{Example of curves violating the Hasse Principle}\label{sec:examples}
The real quadratic number field $K_{48035713} = K_{229,3}$ has ideal class group with $3$-rank $r_{3}(D) = 2$, therefore, Proposition~\ref{Selmer} gives us that $d(\mathcal{S}_{\lambda}) = 2$. On the other hand, the rank of the elliptic curve $E_{D}$ was computed in MAGMA and PARI/GP and was found equal to $1$. Hence, the short exact sequence (\ref{splitrank}) and Lemma~\ref{Pmn} imply that the rank of the quotient group $E_{D'}(\mathbb{Q})/\lambda(E_{D}(\mathbb{Q}))$ is exactly equal to $1$, and the short exact sequence (\ref{selsha}) implies that the rank of $\Sha(E_{D})[\lambda]$ should also be equal to $1$. Denote by $\langle [\xi_{1}]\rangle \subseteq \mathcal{S}_{\lambda}$ the subgroup corresponding to the solvable part of $\mathcal{S}_{\lambda}$. The group $\langle [\xi_{1}]\rangle$ defines the only cubic field, up to conjugation, with an Index Form in its $SL(2,\mathbb{Z})$-equivalence class whose leading coefficient is a cube. If any one of the other three, $3 = (\frac{3^{r_{3}(D)}-1}{2} - 1)$, remaining non-conjugate cubic fields would have such a form, then the subgroup of $\mathcal{S}_{\lambda}$ corresponding to this field must also have a form with leading coefficient a cube, implying that the rank of $E_{D'}(\mathbb{Q})/\lambda(E_{D}(\mathbb{Q}))$ and therefore of $E_{D}$ must be at least $2$, by Lemma~\ref{lem:IndependencePnts}, a contradiction. 

Let us denote by $L^{i}, 1 \leq i \leq 4$, the $4$ non-conjugate cubic fields of discriminant $D$. Let $L^{1}$ be the one with defining polynomial the monic polynomial $f_{229,3}(x) = x^3 - 229x +3$, equivalently with Index Form $[1,0,-m,n]$. Apart from $f_{229,3}(x)$, as expected, the remaining three cubic polynomials that we computed in PARI/GP were monic but of discriminant $t^2D$ for some $0 \neq t \in \mathbb{N}\setminus \mathbb{N}^{3}$. We followed the procedure outlined in \cite[Proposition 2.1]{BelCoh} and we computed the four binary cubic forms $F_{i}(x,y)$ corresponding to $L^{i}$:
$$F_{1}(x,y) = x^3-229xy^2+3y^3,$$
$$F_{2}(x,y) = -134x^3 + 45yx^2 + 41y^2x - 2y^3,$$ 
$$F_{3}(x,y) = -19x^3 + 16yx^2 + 83y^2x - 7y^3,$$
$$F_{4}(x,y) = 23x^3 + 20yx^2 - 75y^2x - 17y^3.$$

The curve $C_{1}: F_{1}(x,y) = z^3$ has at least two rational points, namely $Q_{1}=(1 : 0 : 1)$ and $Q_{m/n}=(1 : m/n : 1)$. The point $Q_{1}$ is the point $Q_{a}$ we defined above, with $a=1$ since $F_{1}$ is monic. 

The other three curves $C_{i}: F_{i}(x,y) = z^3, \ 2 \leq i \leq 4$,  as we already mentioned above, cannot have any rational points but since they correspond to elements in $\mathcal{S}_{\lambda}$, they violate the Hasse principle. 

We remark that any other Torsor $C_{F}: F(x,y) = z^{3}$, with $F$ belonging to the $SL(2,\mathbb{Z})$-equivalence class of either $F_{2}, F_{3}$, or $F_{4}$, cannot have any rational points either, and therefore violates the Hasse Principle as well. We also note that these three classes $[C_{2}], [C_{3}],$ and $[C_{4}]$, even though they correspond to non-conjugate cubic fields, they cannot belong to three independent classes in $\Sha(E_{D})[\lambda]$, since the dimension of this $\mathbb{F}_{3}$-vector space is~$1$. The reason two distinct classes of Torsors might be mapped to the same class in $\Sha(E_{D})[\lambda]$, is because they might become equivalent modulo the solvable part of $\mathcal{S}_{\lambda}$, as one can infer from the short exact sequence (\ref{selsha}).

\section{Final Remarks}
In the diagaram below, we summarize some of the bijections that were discussed or established in the previous sections. Only bijections $(2)$ and $(5)$ were not discussed before, but these follow directly from Class Field Theory and Kummer Theory respectively.
\vspace{0.5cm}

\small
\begin{tikzpicture}[>=latex] 
\node[myboxrectangle] (Selmer) {$\langle [\xi] \rangle \subseteq \mathcal{S}_{\lambda}$};
\node[myboxrectangle] (LD) [right =of Selmer] {$L_{[\mu]}$ \\ $L_{[\mu]}/K_{D}$ unramified} edge [<->] (Selmer);
\node[myboxrectangle] (KD) [right =of LD] {$\langle[\mathfrak{a}_{D}]\rangle \in Cl(K_{D})_{3}/(Cl(K_{D})_{3})^{3}$} edge [<->] (LD);
\node[myboxrectangle] (KD') [below =of KD] {$\langle[\mathfrak{a}_{D'}]\rangle \in Sel_{3}(K_{D'})$} edge [<->] (KD);
\node[myboxrectangle] (3Units) [below =of KD'] {Class of $3$-virtual units \\ $[\mu] \in K_{D'}^{\times}/(K_{D'}^{\times})^{3}$} edge [<->] (KD');
\node[myboxrectangle] (IndexForms) [below =of LD] {$SL(2,\mathbb{Z})$-equivalence classes \\ $[F] \in IF_{L_{[\mu]}}(D)$} edge [<->] (LD);
\node[myboxrectangle] (Torsors) [below =of IndexForms] {Classes of $E_{D'}$-Torsors $[C]$ \\  $C: F(x,y)=z^{3}$} edge [<->](IndexForms);
%\draw [<->] (IndexForms) -- (3Units) ;
    \draw [-latex] (Selmer) --  node [above] {\tiny{(1)}} (LD) ;
      \draw [-latex] (LD) --  node [above] {\tiny{(2)}} (KD) ;
            \draw [-latex] (KD) --  node [right] {\tiny{(5)}} (KD') ;
      \draw [-latex] (KD') --  node [right] {\tiny{(6)}} (3Units) ;
      %\draw [-latex] (3Units) --  node [above] {\tiny{(5)}} (IndexForms) ;
            \draw [-latex] (LD) -- node [right] {\tiny{(3)}} (IndexForms) ;
             \draw [-latex] (IndexForms) -- node [right] {\tiny{(4)}} (Torsors) ;
              \path[<->] [bend left=5]
              (LD) edge node [right] {\tiny{(7)}} (3Units);
\end{tikzpicture}
\normalsize 
\vspace{0.5cm}

Our final remark concerns the form of the discriminants. Let us consider instead of our discriminants $D_{m,n}$, the general case of the Honda discriminants $\Delta = \mathfrak{f}^{2}disc(K_{\alpha,\beta})$ discussed in Section~\ref{subs:discs}. We can easily see that the subfamily of these discriminants satisfying $6 \nmid \Delta$ and $\Delta \equiv \pm 4 \bmod 9$, would still satisfy Propositions~\ref{Selmer}, \ref{prop:PntsCubicForm} and \ref{prop:CubeCoeff}, Lemma~\ref{PointForm}, the equivalence result~(\ref{parity}), and Corollary~\ref{ParityRanksDiagram}. We notice that in this case, and for $\Delta < 0$, the rank of the elliptic curves $E_{\Delta}$ might be zero, whereas for $\Delta > 0$, we must always have at least one rational point, always by assuming finiteness of the Tate-Shafarevich group. Furthermore, there is no monogenic field amongst the unramified cubic extensions of $K_{\Delta}$, equivalently $\Delta \notin \mathcal{D}$, if and only if $E_{\Delta'}$ has no \emph{integral} points. And therefore, in the case of $\Delta > 0$, the existence of no monogenic fields amongst the unramified cubic extensions of a quadratic number field $K_{\Delta}$, would yield elliptic curves $E_{\Delta'}$ with nontrivial rank, yet no integral points. 

\section*{Acknowledgements}
The author is grateful to the anonymous Referee for their careful review of the paper, and for providing valuable comments and suggestions that improved the quality of this manuscript. The author would like to thank Antoine Joux and the CISPA Helmholtz Center for Information Security, as well as the Max Planck Institute for Mathematics in Bonn, for their hospitality and support during the writing of this paper. The author would also like to thank Evis Ieronymou for interesting discussions, and the University of Cyprus where she spent two weeks as a research visitor while working on this paper.

\begin{figure}[!b]
    \includegraphics[width = 0.12\textwidth]{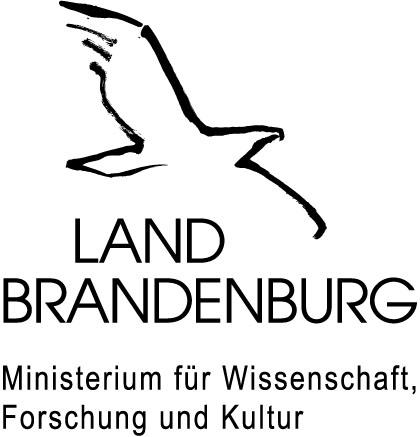}
\end{figure}

\end{document}